\documentclass[a4paper,10pt]{article}
\usepackage[utf8]{inputenc}
\usepackage{amsmath}
\usepackage{amsfonts}
\usepackage{amssymb}
\usepackage{amsthm}
\usepackage{graphicx}
\usepackage{xspace}
\usepackage{cite}
\usepackage[margin=1.32in]{geometry}

\newtheorem{theorem}{Theorem}[section]
\newtheorem{lemma}{Lemma}[section]
\newtheorem{proposition}{Proposition}[section]
\newtheorem{corollary}{Corollary}[section]
\newtheorem{definition}{Definition}[section]
\newtheorem{conjecture}{Conjecture}[section]
\theoremstyle{definition}
\newtheorem{remark}{Remark}[section]
\newcommand{\Ad}{\textup{\textrm{Ad}}\xspace}
\newcommand{\ad}{\textup{\textrm{ad}}\xspace}

\newcommand{\Lie}{\textup{\textrm{Lie}}\xspace}
\newcommand{\rspan}{\textup{$\mathbb{R}$\textrm{-span}}\xspace}
\newcommand{\zspan}{\textup{$\mathbb{Z}$\textrm{-span}}\xspace}
\newcommand{\rank}{\textup{\textrm{rank}}\xspace}
\newcommand{\Vol}{\textup{\textrm{Vol}}\xspace}
\newcommand{\SL}{\textup{\textrm{SL}}\xspace}
\newcommand{\GL}{\textup{\textrm{GL}}\xspace}
\newcommand{\PN}{\textup{\textrm{$\mathcal{KN}$}}\xspace}
\newcommand{\CPN}{\textup{\textrm{$\mathcal{KN}$}}\xspace}
\newcommand{\hull}{\textup{\textrm{Hull}}\xspace}

\usepackage[pdftex,pagebackref,letterpaper=true,colorlinks=true,pdfpagemode=none,urlcolor=blue,linkcolor=blue,citecolor=blue,pdfstartview=FitH]{hyperref}
\usepackage[english]{babel}

\title{Non-Divergence of Unipotent Flows on Quotients of Rank One Semisimple Groups}
\author{C. Davis Buenger and Cheng Zheng }

\begin{document}

\allowdisplaybreaks
\maketitle

\begin{abstract}
Let $G$ be a semisimple Lie group of rank $1$ and $\Gamma$ be a torsion free discrete subgroup of $G$.   
We show that in $G/\Gamma$, given $\epsilon>0$, any trajectory of a unipotent flow remains in the set of points with injectivity radius  larger than $
\delta$ for
$1-\epsilon$ proportion of the  time for some $\delta>0$. The result also holds for any finitely generated 
discrete
subgroup $\Gamma$ and this generalizes Dani's quantitative nondivergence theorem \cite{D} for lattices of
rank one semisimple groups. Furthermore, for a fixed $\epsilon>0$ there exists an injectivity radius $\delta$ such that for any unipotent trajectory $\{u_tx\}_{t\in [0,T]}$, either it spends at least $1-\epsilon$ proportion of the time in the set  with injectivity radius larger than $\delta$ for all large $T>0$ or there exists a $\{u_t\}_{t\in\mathbb{R}}$-normalized abelian subgroup $L$ of $G$ which intersects $g\Gamma g^{-1}$ in a small covolume lattice.  
We also extend these results when $G$ is the product of rank-$1$ semisimple
groups and $\Gamma$ a discrete subgroup of $G$ whose projection onto each nontrivial factor is torsion free. 
\end{abstract}

\section{Background and Statements}
Let $G$ be a  Lie group, $\Gamma$ be a discrete subgroup of $G$, and $\{u_t\}_{t
\in\mathbb{R}}$ be a
one-parameter unipotent subgroup in $G$ (that is, $Ad_{u_t}$ is unipotent). In the case that $G=\SL(n,\mathbb{R})$ and $\Gamma=\SL(n,\mathbb{Z})$, Margulis \cite{M} proved that each unipotent trajectory is non-divergent. Later, Dani \cite{D79} improved Margulis's result  by showing that for each unipotent trajectory on $\SL(n,\mathbb{R})/\SL(n,\mathbb{Z})$ there exists a compact set $K$ such that the relative  time the trajectory spends in $K$ is of positive proportion. In \cite{D}, Dani further improved this result 
for any $\mathbb{R}$-rank 1 Lie group and $\Gamma$ an arbitrary lattice in $G$ and showed that for any unipotent trajectory  and any $\epsilon>0$ there exists a compact set $K(\epsilon)$ such that the relative proportion of time the trajectory spends in $K(\epsilon)$ is at least $1-\epsilon$ . Finally in \cite{D2}, Dani generalized the theorem for any Lie group $G$ and any lattice $\Gamma$ in $G$.



In \cite[page 230]{Rat} Ratner defined divergence for orbits on quotients of Lie groups by discrete groups as follows: given a Lie group $G$, a discrete group $\Gamma$, $h\in G$ and $x=g\Gamma \in G/\Gamma$, the $h$-orbit of $x$ is said to be divergent if there exists a sequence $\{\gamma_n\}\in \Gamma\setminus \{e\}$ such that $(h^ng)\gamma_n(h^ng)^{-1}\rightarrow e$ as $n\rightarrow\infty$. Note that if $\Gamma$ is a lattice, then  $h^nx$ is divergent if and only if for every compact set $K$, we have $h^nx\not\in K$ for all large $n$\cite[Chap.1]{R}.
In light of this definition, Ratner asked if the following holds:

\begin{conjecture} Given a unipotent one-parameter subgroup $\{u_t\}_{t\in\mathbb{R}}\subseteq G$ and a discrete subgroup $\Gamma$ of $G$, for any point $x=g\Gamma\in G/\Gamma$, there exists a neighborhood $N$ of the identity in $G$ such that 
\begin{equation*}
 \frac{1}{T}m(\{t\in[0,T]:N\cap(u_tg)\Gamma(u_{t}g)^{-1}\neq \{e\}\})<\epsilon  \text{ for all large $T>0$},
\end{equation*} 
where $m$ is the Lebesgue measure on $\mathbb{R}$.
\end{conjecture}

We address this conjecture when  $G$ is a semisimple $\mathbb{R}$-rank 1 group and more generally when G is the product of semisimple $\mathbb{R}$-rank 1 groups. For a Lie group $G$ and discrete group $\Gamma$, let $X_\delta$ denote the set of points in $G/\Gamma$ with injectivity radius at least $\delta$ (see Section~\ref{sec:two}).

\begin{theorem}\label{maintheorem1}
Suppose $G$ is a semisimple Lie group of $\mathbb{R}$-rank 1 and $\Gamma$ is a virtually torsion-free discrete subgroup of 
 $G$. Let $X=G/\Gamma$. Then there exists a constant $c_\epsilon>0$
 such that for any one-parameter unipotent subgroup $\{u_t\}_{t\in\mathbb{R}}$  of $G$, $\epsilon>0$, 
 $\delta>0$, $x\in X_{\delta}$, and $T>0$ 
 \begin{equation}\label{thm1}
   \frac{1}{T}m(\{t\in[0,T]:u_tx\notin X_{c_\epsilon\delta}\}) \leqslant\epsilon.
 \end{equation}
\end{theorem}
\begin{remark} 
Note that all finitely generated groups in $\GL(n,\mathbb{C})$ are virtually torsion free by a theorem of Selberg \cite{Sel}.
\end{remark}

We also address an analogue of the following uniform non-divergence result due to Dani:
\begin{theorem}[Therorem 2.1, \cite{D2}] Let $n\in\mathbb{N}$ and $\epsilon>0$ be given. Then there exists $\delta>0$ such that for any unipotent one-parameter subgroup $\{u_t\}_{t\in\mathbb{R}}$ in $\SL(n,\mathbb{R})$ and any lattice $\Lambda\in\mathbb{R}^n$ with $covolume$ $1$, the following holds: either
\begin{equation}m(\{t\in[0,T]:u_tg\Lambda\cap B_\delta=\{0\}\})>(1-\epsilon)T \text{ for all large $T>0$}
\end{equation}
or there exists a $\{u_t\}_{t\in\mathbb{R}}$-invariant proper non-zero subspace $W$ of $\mathbb{R}^n$ such that $W\cap\Lambda$ is a lattice in $W$.
\end{theorem}

Dani's result was later quatified by Kleinbock and Margulis~\cite{KM} and Kleinbock~\cite{K}. In light of their uniform results, we have the following dichotomy:
\begin{theorem}\label{uniform}
Let $G$ be a semisimple Lie group of $\mathbb{R}$-rank 1 and $\Gamma$ be a virtually torsion-free discrete group in $G$.  Given $\epsilon>0$ there exist positive computable constants $\delta_\epsilon$ and $\beta_\epsilon$ such that for any $\delta<\delta_\epsilon$, any one-parameter unipotent subgroup $\{u_t\}_{t\in\mathbb{R}}$ of $G$, and any $x=g\Gamma \in G/\Gamma$  either 
\begin{equation*}
\frac{1}{T}m\left(\left\{t\in[0,T]: u_tx\not\in X_\delta\right\}\right)<\epsilon \text{ for all large $T>0$}
\end{equation*}
or there exists a proper abelian subgroup $L$ of $G$ such that $\Delta:=g\Gamma g^{-1}\cap L$ is a lattice in $L$, $\{u_t\}_{t\in\mathbb{R}}$ normalizes $L$, and the covolume of $u_t\Delta u_{-t}$ in $u_tLu_{-t}=L$ is a fixed constant less than $(\beta_\epsilon\delta)^{\dim(L)} $ for all $t\geqslant0$, where  we define the measure on $L/\Delta$ as the measure induced by a fixed right invariant inner product on  $\Lie(G)$.
\end{theorem}

Similar results hold in the case when $G$ is the product of $\mathbb{R}$-rank 1 Lie groups:

\begin{theorem}\label{maintheorem2}
 Suppose  $G=G_1\times G_2\times\dots\times G_n$ where each $G_i$ is semisimple and of $\mathbb{R}$-rank $1$. 
 Let $\Gamma$ be a discrete subgroup of $G$ such that the projection of $\Gamma$ onto each coordinate does not contain nontrivial elliptic elements.
Then the conclusions of Theorems \ref{maintheorem1} and \ref{uniform} hold.
\end{theorem}

In fact, from our proof we can conclude the following

\begin{theorem}
Suppose  $G=\SL(2,\mathbb{C})\times \SL(2,\mathbb{C})\times\dots\times \SL(2,\mathbb{C})$.
 Let $\Gamma$ be any discrete subgroup of $G$. Then the conclusions of Theorems \ref{maintheorem1} and \ref{uniform} hold.\end{theorem}
\begin{remark}Our proof uses the existence of a Zassenhaus neighborhood, the  Lie algebra structure, and the Bruhat decomposition for $\mathbb{R}$-rank 1 semisimple Lie groups. We make direct use of the result of Kleinbock and Margulis on quantitative non-divergence on the space of unimodular lattices in $\mathbb{R}^n$, where $n$ is the dimension of $G$. Unlike the proof of Dani's theorem, we do not make use of the structure theorem for the fundamental domain of lattices in $\mathbb{R}$-rank 1 semisimple Lie groups.
\end{remark}
\section*{Acknowledgement} We would like to thank  Nimish Shah and Dimitry Kleinbock for their very helpful suggestions
and comments during our discussions.

\section{Notations and Preliminaries} \label{sec:two}

Let   $G$ be a  semisimple Lie group of $\mathbb{R}$-rank 1  and let $\mathfrak g$ be its Lie algebra. 
Let $K$ be a maximal compact subgroup of $G$ and let $\mathfrak k$ be it's Lie algebra. Let $\langle\cdot,\cdot
\rangle$ denote the Killing form on $\mathfrak g$ and let $\mathfrak p$ be the orthogonal complement to $\mathfrak 
k$ with respect to the Killing form. We have 
the Cartan decomposition 
$$\mathfrak g=\mathfrak k\oplus\mathfrak p.$$
Fix a one dimensional subspace $\mathfrak{a}$ of $\mathfrak{p}$.  
Consider the root space decomposition of $\mathfrak{g}$ with respect to $\mathfrak a$.  For a root $\beta\in\mathfrak a^*$, let $
\mathfrak {u}^\beta$ denote the root space associated to $\beta$ and let $\mathfrak z$ denote the space $\mathfrak 
{u}^0$. Since $\mathfrak g$ is of $\mathbb{R}$-rank 1,
 there exists a real root $\alpha$ such that:
$$\mathfrak g=\mathfrak {u}^{-2\alpha}\oplus\mathfrak {u}^{-\alpha}\oplus\mathfrak z\oplus\mathfrak {u}^\alpha
\oplus\mathfrak {u}^{2\alpha}.$$

Let $\mathfrak{n}^+=\mathfrak {u}^\alpha\oplus\mathfrak {u}^{2\alpha}$ and $\mathfrak n^-=\mathfrak {u}
^{-\alpha}\oplus\mathfrak {u}^{-2\alpha}$. We define $N^{+}=\exp(\mathfrak {n}^{+})
$ and $N^{-}=\exp(\mathfrak {n}^{-})$. Throughout we will abbreviate $N^+$ by 
$N$ and $\mathfrak{n}^{+}$ by $\mathfrak{n}$. Let $\mathfrak m=\mathfrak z\cap\mathfrak k$. Then $$\mathfrak z=\mathfrak a\oplus\mathfrak m.$$

Let $A=\exp(\mathfrak a)$ and $M=Z_G(A)\cap K$. Then $\Lie(M)=\mathfrak m$.
Let $W$ be the Weyl group for $G$.  As $G$ is of $\mathbb{R}$-rank 1, $W$ is a group of order two. Let $\tilde{\omega}\neq e
\in W$, and fix a representative  $\omega\in G$ of $\tilde{\omega}$. We 
have the following Bruhat decomposition for semisimple groups of $\mathbb{R}$-rank 1 \cite[Section 12.14]{R}
\begin{equation}\label{bruhat}G=MAN\bigsqcup MAN\omega MAN.\end{equation}

Let $\Theta$ be the involution on $\mathfrak g$ corresponding to our Cartan decomposition and define positive 
definite inner product $(\cdot,\cdot)$ on $\mathfrak g$ by 
$$(x,y)=-\langle x,\Theta(y)\rangle.$$
Define $$\|x\|:=(x,x)^{\frac12}.$$
By definition, the norm $\|x\|$ is $\Ad_k$ invariant for all $k\in K$.\\

A Zaussenhaus neighborhood, $\Omega$, for a Lie  group $F$ is defined as a neighborhood of $e\in F$ such that for any discrete group $\Delta\subseteq F$, there exists a  connected nilpotent subgroup $H\subseteq F$ such that $\Delta\cap \Omega\subseteq H$. Note that any Lie group $G$ admits a Zaussenhaus neighborhood \cite[Theorem 8.16]{R}.\\

Let $\mathfrak b_r$ denote the open ball of radius $r$ 
centered at the origin in $\mathfrak g$. Let $0<r_0<\frac{1}{k}$, where $k$ is the dimension of $G$,  be  such that the exponential map homeomorphically maps  $
\mathfrak b_{r_0}$ into $G$ and $\exp(\mathfrak b_{r_0})$ is a Zaussenhaus neighborhood for $G$.  For $\eta\leqslant r_0$, 
denote $\exp(\mathfrak b_\eta)$ as $B_\eta$.\\

Let $\Gamma$ be a discrete subgroup of $G$ and let $X=G/\Gamma$. As usual, $G$ acts on $X$ by left 
translations.  For $x\in X$ let  $G_x:=\{g\in G: gx=x\}$.  Then $G_x$ is a conjugate of $\Gamma$. 
For any $0<\eta\leqslant r_0$, define 
\begin{equation*}
X_{\eta}:=\left\{x\in X: G_x\cap B_{\eta}=\{e\}\right\}, 
\end{equation*}
and for $\eta>r_0$ define $X_{\eta}=X$.
Using the notation and definitions above, inequality (\ref{thm1}) is equivalent to:
 \begin{equation}\label{thm3}
  m\left(\left\{t\in[0,T]:G_{u_tx}\cap B_\delta\neq \{e\}\right\}\right)\leqslant\epsilon T.
 \end{equation}
\section{Uniform Virtual Linearization}

In this section, we show that 
a unipotent trajectory lying in the complement of $X_{r_0}$ can be associated to a unipotent trajectory in $\SL(n,\mathbb{R})/\SL(n,\mathbb{Z})$, where $n$ is the dimension of $G$. This association allows us to apply the quantitative non-divergence results of Dani, Kleinbock, and Margulis to prove our result.
 Throughout this section, we assume that $G$ is a semisimple Lie group 
of $\mathbb{R}$-rank $1$. 
\subsection{Algebraic properties of rank-$1$ semisimple Lie groups}

\begin{lemma}
Any nilpotent subgroup of $G$ with no nontrivial elliptic elements is contained in a conjugate of either $MA$ or $MN
$.
\end{lemma}
\begin{proof}
 Let $U$ be a nilpotent group with no nontrivial elliptic elements. Let $x$ be a nontrivial element in the center of $U$. By assumption $x$ is not elliptic.\\
 
 By the multiplicative Jordan decomposition, there exist three commuting elements $k$, $h$ and $u$ in 
 $G$ such that $x=khu$
 and $k$ is elliptic, $h$ is hyperbolic and $u$ is unipotent. Moreover, $k$, $h$ and $u$ commute with every 
 element commuting with $x$.\\
 
 If $h\neq e$, then $h$ is conjugate to an element in $A$. For simplicity, assume $h\in A$. All elements commuting with $h$ are in $MA$, since $G$ is $\mathbb{R}$-rank 1. By construction of $h$, all elements commuting with $x$ also commute with $h$, so $U$ is contained in $MA$.\\
 
 If $h=e$, since $x$ is not eliptic, $u\neq e$, and $u$ has a conjugate in $N$. For simplicity, we assume $u\in N$. We shall show that
 all elements commuting with $u$ are in $MN$. Suppose $g$ commutes with $u$. By the Bruhat decomposition, $g\in MAN\omega MAN$ or  $g\in MAN$.\\
  
  If $g\in MAN \omega MAN$, then there exist $g_1,g_2\in MAN$ such that $g=g_1\omega g_2$. We have $g_1\omega 
  g_2u=ug_1\omega g_2$
  and $\omega g_2ug_2^{-1}\omega^{-1}=g_1^{-1}ug_1$. Since $MAN$ normalizes $N$ and conjugation by $\omega$ sends 
  $N$ to $N^{-}$, we have
  $g_2ug_2^{-1}=g_1^{-1}ug_1\in N\cap N^-=\{e\}$. By assumption $u\neq e$, so this case cannot occur.\\
  
  If $g\in MAN$, then there exist $m$, $a$, and $n$ such that $g=man$.  Since $N$ is simply connected the exponential map  from $\mathfrak{n}=\mathfrak{u}^{\alpha} 
  \oplus\mathfrak{u}^{2\alpha}$ to $N$ is bijective. Thus there exists $u_\alpha\in\mathfrak{u}^{\alpha}$, $u_{2\alpha}\in\mathfrak{u}^{2\alpha}$ and $n_0\in\mathfrak n$ such that  $u=\exp(u_\alpha+u_{2\alpha})$ and $n=\exp(n_0)$. We have that
  \begin{eqnarray*}
   manun^{-1}a^{-1}m^{-1} &=& u\\
   \Ad(ma)\Ad(n)(u_\alpha+u_{2\alpha}) &=& u_\alpha+u_{2\alpha}\\
   \Ad(ma)[\exp(\ad(n_0))(u_\alpha+u_{2\alpha})] &=& u_\alpha+u_{2\alpha}\\
   \Ad(ma)(u_\alpha+u_{2\alpha}+\ad(n_0)(u_\alpha)) &=& u_\alpha+u_{2\alpha}.
  \end{eqnarray*}
  $MA$ normalizes $\mathfrak{u}^{\alpha}$ and $\mathfrak{u}^{2\alpha}$ respectively, so 
  \begin{eqnarray}\label{eqnl31}
   \Ad(ma)(u_\alpha) &=& u_\alpha\\
  \label{eqnl312} \Ad(ma)(u_{2\alpha}+\ad(n_0)(u_\alpha)) &=& u_{2\alpha}.
  \end{eqnarray}
To show that $g\in MN$, we have to show $a=e$. Suppose that $a\neq e$. Then from equation ($\ref{eqnl31}$) and the fact that our norm is $\Ad_m$ invariant, we get
  \begin{equation*}
   \Vert u_\alpha\Vert=\Vert \Ad(m)\Ad(a)(u_\alpha)\Vert=\Vert \Ad(a)(u_\alpha)\Vert=|\alpha(a)|\Vert u_\alpha\Vert.
  \end{equation*}
 Since $a\neq e$, $|\alpha(a)|\neq 1$, so $u_\alpha=0$. Then by the same argument with $u_\alpha=0$,
  we get from equation (\ref{eqnl312}) that $u_{2\alpha}=0$. This forces $u$ to be $e$. However, by assumption $u\neq e$.
  Therefore $g\in MN$.\\
  
  Thus all elements commuting with $u$ are in $MN$. In particular $ U$ is contained in $MN$.
\end{proof}

\begin{remark} A maximal compact abelian subgroup of $K$ may not be contained in $M$.  Thus, the above lemma fails if we allow nilpotent subgroups with non-trivial elliptic elements. As an example, consider $SO(4,1)$. In this case, $K$ is isomorphic to $SO(4)$ and $M$ is isomorphic to $SO(3)$. $SO(4)$ contains a two dimensional torus, while any abelian subgroup of $SO(3)$ is either 0 or 1 dimensional. 
\end{remark}

\begin{corollary}\label{l31}For any torsion-free discrete subgroup $\Lambda$ of $G$, $\Lambda\cap B_{r_0}$ is 
contained in a conjugate of either $MA$ or $MN$, where $B_{r_0}$ is a Zassenhaus neighborhood.
\end{corollary}
\begin{proof}
$B_{r_0}$ is a Zassenhaus neighborhood, so the group generated by
 $\Lambda\cap B_{r_0}$ is a torsion free discrete nilpotent subgroup, and the claim follows from Lemma \ref{l31}.
 \end{proof}
 
\begin{lemma}\label{l32}
 Let $g$ be an element in $G$. We have
  \begin{description}
   \item[(1)] $gMNg^{-1}\cap MA\subseteq M$.
   \item[(2)] If $gMAg^{-1}\cap MA\nsubseteq M$, then $gMAg^{-1}=MA$ and $g\in MA\cup \omega MA$.
   \item[(3)] If $gMNg^{-1}\cap MN$ contains a non-elliptic element, then $gMNg^{-1}=MN$ and in the case that $g$ is unipotent, $g\in N$.
  \end{description}
\end{lemma}
\begin{proof} We shall use the Bruhat decomposition for $G$. For each claim, we will first consider $g\in MAN$ and then consider $g\in MAN\omega MAN$.\\

  \item[(1)] Let 
$g\in MAN$. Since $MAN$ normalizes $MN$, $$gMNg^{-1}\cap MA=MN\cap MA=M.$$ 
Now let $g\in MAN\omega MAN$. $g=g_1\omega g_2$ for some $g_1,g_2\in MAN$.  Suppose  $y\in MA$  and $x\in 
   MN$ are such that $gxg^{-1}=y$. Then
$\omega^{-1}g_2xg_2^{-1}\omega=g_1^{-1} yg_1$. The left hand side belongs to $MN^{-}$, the right hand side 
belongs to $MAN$. Thus $g_1^{-1} yg_1\in MN^{-1}\cap MAN=M$ and $y$ must have trivial $A$ component. \\

  \item[(2)] Let $x\in MA$ and $y\in MA\setminus M$ be such that $gxg^{-1}=y$. Since the only elliptic elements in $MA$ are in $M$, we may assume that $x\in MA\setminus M$ as well. \\

 Let $g\in MAN$ Then $g=man$ for some $m\in M$, $a\in A$ and $n\in N$. We have $manxn^{-1}a^{-1}m^{-1}=y$. Then $nxn^{-1}\in MA$ and hence $x^{-1}nxn^{-1}\in MA$. However, $(x^
   {-1}nx)n^{-1}\in N$
   since $MA$ normalizes $N$. So $x^{-1}nxn^{-1}=e$ i.e. $x$ and $n$ commute. This forces $n$ to be identity as 
   $x\in MA\setminus  M$. Thus $g\in MA$ and $gMAg^{-1}=MA$.\\

Now let $g\in MAN\omega MAN$. Then $g=g_1\omega g_2$ for some $g_1,g_2\in MAN$. We have 
$$g_1\omega g_2 x g_2^{-1}\omega^{-1}g_1^{-1}=y.$$ Let $\beta_1=g_1^{-1} 
   y g_1\in MAN$ and 
   $\beta_2=g_2 x g_2^{-1}\in MAN$. Then $\omega\beta_2\omega^{-1}=\beta_1$. Conjugation by the Weyl 
   group element $\omega$ sends
   $N$ to $N^{-}$. Therefore, both $\beta_1$ and $\beta_2$ are elements in $MA$. Now $\beta_1=g_1^{-1}yg_1$ with $y\in MA\setminus M$ and $\beta_1\in MA$. By repeating the argument from the case when $g\in MAN$, it follows that $g_1\in MA$. In the same way, $g_2\in MA$.
   Since conjugation by $\omega$ sends $MA$ to itself, $gMAg^{-1}=MA$.\\

  \item[(3)]
  Let $g\in MAN$. Since $MAN$ normalizes $MN$, $gMNg^{-1}=MN$. If $g$ is unipotent, then $g\in N$.
\\

Now let  $g\in MAN\omega MAN$. Then $g=g_1\omega g_2$ for some $g_1,g_2\in MAN$. We can find non-elliptic
   elements $x,y\in MN$ such that
   $g_1\omega g_2 x g_2^{-1}\omega^{-1} g_1^{-1}=y$, i.e. $\omega g_2 x g_2^{-1}\omega^{-1}=g_1^{-1}yg_1$. 
   Since $g_1,g_2$ normalize
   $MN$ and $\omega$ sends $N$ to $N^{-1}$, $g_1^{-1}yg_1$ has to be an element in $M$. This contradicts the 
   assumption that $y$ is
   non-elliptic.\end{proof}
\begin{corollary}\label{c32} Let $H_1$ and $H_2$ be subgroups of $G$ such that each is conjugate to either $MA$ or $MN$. If there exists a non-elliptic element  $\gamma$  in $H_1\cap H_2$, then $H_1=H_2$. 
\end{corollary}
\begin{proof}  By (1) $H_1$ and $H_2$ must either both be conjugate to $MA$ or both be conjugate to $MN$. If they are both conjugate to $MA$, then (2) implies $H_1=H_2$.  If they are both conjugate to $MN$, then (3) implies $H_1=H_2$.
\end{proof}
The following observation is crucial to our proof. 
\begin{lemma}\label{l33}  Let $F$ be a Lie group with Lie subgroups $H$ and $K$  such that $K$ is compact, $K$ normalizes $H$, and $F\cong K\ltimes H$.
Suppose there exists an automorphism $\theta$ of $F$ contracting $H$ and fixing every element of $K$, where by contracting, we mean that $\theta^k(h)\rightarrow e$ as $k\rightarrow \infty$ for all $h\in H$. 
Then there exists a constant $l_K$ depending only on $K$ such that for all discrete subgroups 
 $
\Lambda\subseteq F$, there exists a finite index subgroup $\Lambda_0$ 
of $\Lambda$ such that $[\Lambda:\Lambda_0]\leqslant l_K$ and $\Lambda_0$ is contained in a connected nilpotent 
subgroup of $F$.
\end{lemma}
\begin{proof} Let $\Lambda$ be a discrete subgroup of $F$.
For any $x\in F$, we write $k(x)$ for its $K$ component and $h(x)$ for its $H$ component; in other words,
  $x=(k(x),h(x))\in KH$. \\
  
 Fix a neighborhood $U$ of $e$ in $K$ and a neighborhood $V$ of $e$ in $H$ such that $U\times V$ is a 
  Zassenhaus neighborhood for
  $F$. 
  Let $S=\{\lambda\in\Lambda|k(\lambda)\in U\}$
  and define $\Lambda_0$ as the subgroup of $\Lambda$ generated by $S$. We show that $\Lambda_0$ satisfies 
  the required properties.\\
  
  We first prove that $\Lambda_0$ is contained in a connected nilpotent subgroup of $F$. 
  Let $\{\gamma_i\}_{i=1}^\infty$ be an indexing of the elements of $S$ and fix $n\in \mathbb{N}$.
  Since $\theta$ contracts $N$, there exists $m\in\mathbb{N}$ such that  $\theta^m(h_i)\in V$ for all $i\in\{1,\dots,n\}$.  
  Furthermore, for  each $i\in\{1,\dots,n\}$,  $\lambda_i'=\theta^m(k_ih_i)=k_i\theta^m(h_i)$ is in the Zassenhaus neighborhood $U\times V$, and as $\theta$ is a 
  contracting automorphism the group generated by  $\{\lambda_i'\}_{i=1}^n$ is still discrete. Thus there exists a connected nilpotent subgroup $H_n'$ containing $\{\lambda_i'\}_{i=1}^n$.  After applying negative powers of $\theta$,  $\{\gamma_i\}_{i=1}^n$ is contained in a connected nilpotent
  subgroup $H_n=\theta^{-m}(H_n')$. We denote its corresponding nilpotent subalgebra by $\mathfrak{h}_n$.\\
  
By the construction above, this sequence $\mathfrak{h}_1,\mathfrak{h}_2,\dots$ is increasing
(i.e.  $\mathfrak{h}_1\subseteq\mathfrak{h}_2\subseteq...$). Let $\mathfrak{h}=\bigcup_{k=1}^\infty\mathfrak{h}_k$ 
  and $F$ be the
  corresponding connected subgroup in $KH$. Then $F$ is a connected nilpotent subgroup and contains $
  \Lambda_0$.\\
  
Now we show that the index $[\Lambda:\Lambda_0]$ is bounded by a number depending only on $K$.
  Fix a neighborhood of $e$
  say $U'\subset U$ such that $U'U'^{-1}\subseteq U$.  Let $l_K=\left\lceil\frac{\Vol{K}}{\Vol{U'}}\right\rceil$.
  If $\lambda_1$,$\lambda_2$,  ...,$\lambda_m$ are different representatives for $\Lambda/\Lambda_0$, then from 
  the definition of $\Lambda_0$,
  we know that $k(\lambda_i)^{-1}k(\lambda_j)\notin U$ for $i\neq j$. This implies that $k(\lambda_i)U'$ are pairwise 
  disjoint in $K$.
  Since $K$ is compact, this forces $m$ to be less than $l_K$. Hence,
  the index $[\Lambda:\Lambda_0]$ is bounded by a number depending only on $K$.\\
\end{proof}

\begin{corollary}\label{c33} 
 Any discrete subgroup $\Lambda$ in $MN$ (respectively $MA$) has a subgroup $\Lambda_0$ of finite index such 
 that
 $\Lambda_0$ is contained in a connected nilpotent subgroup of $MN$ (respectively $MA$). Moreover, the index $
 [\Lambda,\Lambda_0]$
 is bounded by a number depending only on $M$.
\end{corollary}
\begin{proof} Since $M$ is compact and normalizes both $A$ and $N$, by the lemma it will suffice to find  
 automorphisms of $MN$ and $MA$, which act trivially on $M$ and contract $N$ and $A$ respectfully. In the case of $MN$ there exists an element $a\in A$ such that  
conjugation by $a$ is a contracting automorphism of $N$ and $a$ commutes with $M$. In the case of $MA$,
  let $\theta_c:MA \rightarrow MA$ be the  map sending $(m,\exp(tx))$ to $(m,\exp(ctx))$ for $c>0$, where $x$ is a nonzero element in $\mathfrak{a}$. 
  For $c<1$, $\theta_c$ is an automorphism of $MA$ contracting $A$. \\
\end{proof}
Before proving the next lemma, we recall the definition of $k$-step nilpotent group. Let $H$ be a nilpotent group
with $H^{(0)}=H$ and $H^{(i+1)}=[H,H^{(i)}]$. If $k$ is a number such that $H^{(k)}=\{e\}$, then $H$ is called $k$-step 
nilpotent.\\

\begin{lemma}\label{l34}
 Any connected nilpotent subgroup $H$ of $MN$ or $MA$ is 2-step nilpotent. Furthermore, $[H,H]\subseteq [N,N]$.
\end{lemma}
\begin{proof}
Let  $H\subseteq MA$. Since $H$ is connected and nilpotent, the projection of $H$ onto $M$ is compact and 
  nilpotent. However, any
  compact nilpotent Lie group is abelian. So $H$ is abelian, and any abelian group is $2$-step and $[H,H]=\{e\}$.

Now let $H\subseteq MN$.  As G has finite center, by adjoint representation of $G$, it is harmless to assume that $H$, $M$ and $N$ are 
  linear groups 
  in $GL(n,\mathbb{C})$. The projection of $H$ onto $M$ is compact nilpotent, hence abelian.
  So we can assume that $M$ is compact abelian.  Let $\bar{H}$, $\bar{M}$ and $\bar{N}$ be the Zariski closures of 
  $H$, $M$ and $N$ respectively.
  Since $\bar{M}\ltimes\bar{N}$ is algebraic, $\bar{H}\subseteq\bar{M}\ltimes\bar{N}$. Being a nilpotent algebraic 
  group,
  $\bar{H}$ is the product $S\times U$ where $S$ is the maximal torus in $\bar{H}$ and $U$ is the unipotent radical. Moreover, $[\bar{H},\bar{H}]\subseteq [U,U]$. By 
  Jordan decomposition,
  this implies that $U\subseteq\bar{N}$ and $[\bar{H},\bar{H}]\subseteq [\bar{N},\bar{N}]$. Since $N$ is 2-step nilpotent, so is $U$. Therefore, $\bar{H}$ and ${H}$ are 
  2-step nilpotent and $[H, H]\subseteq [N, N]$.

\end{proof}

Another crucial ingredient in our proof is the following observation.
\begin{lemma}\label{l35}
 Let $\Lambda$ be a discrete subgroup of a 2-step connceted nilpotent group $H$. Let $\exp$ denote the exponential map from $\Lie(H)$ to $H$. Then 
 $\zspan\{\exp^{-1}{\Lambda}\}\subseteq\frac{1}{2}\exp^{-1}{\Lambda}$, where $\exp^{-1}$ denotes the pre-image under the exponential map. 
 Moreover, if $\Delta=\zspan\{\exp^{-1}{\Lambda}\}$, then $[\Delta,\Delta]$(= the set of $\mathbb{Z}$-linear combinations of $[x,y]$ where $x,y\in\Delta$) is contained in $\exp^{-1}{[\Lambda,\Lambda]}$.
\end{lemma}
\begin{proof}
 Since $H$ is 2-step nilpotent, by Baker-Campbell-Hausdorff formula,
 \begin{equation*}
\exp X\cdot \exp Y=\exp(X+Y+\frac{1}{2}[X,Y])
\end{equation*}
 for $X,Y$ in the Lie algebra of $H$. Also we have 
 \begin{equation*}
 \exp[X,Y]=[\exp X,\exp Y]
 \end{equation*}
 and elements of this kind are in the center of $H$.
 Now for $X_1,X_2,...,X_n\in \exp^{-1}\Lambda$,
 we have 
 \begin{equation*}
  \exp X_1\cdot \exp X_2\cdot...\cdot \exp X_n=\exp(\sum\limits_{i=1}^k X_i+\frac{1}{2}\sum\limits_{i<j}[X_i,X_j])
 \end{equation*}
 \begin{equation*}
 (\exp X_1\cdot \exp X_2\cdot...\cdot \exp X_n)^2=\exp(2\sum\limits_{i=1}^k X_i+\sum\limits_{i<j}[X_i,X_j])
 \end{equation*}
 \begin{equation*}
 (\exp X_1\cdot \exp X_2\cdot...\cdot \exp X_n)^2\cdot \exp(-\sum\limits_{i<j}[X_i,X_j])=\exp(2\sum\limits_{i=1}^k X_i).
 \end{equation*}
 Therefore $\exp(2\sum\limits_{i=1}^k X_i)\in\Lambda$ and $\sum\limits_{i=1}^k X_i\in\frac{1}{2}\exp^{-1}\Lambda$.
 This finishes the first part of the lemma.\\
 
 Now suppose $X,Y\in \exp^{-1}(\Lambda)$. By the equation $\exp[X,Y]=[\exp X,\exp Y]$, $[X,Y]\in\exp^{-1}([\Lambda,\Lambda])$. Since $\Lambda$ is a discrete subgroup of the $2$-step nilpotent group $H$, $\exp^{-1}{[\Lambda,\Lambda]}$ is an abelian discrete subgroup in the vector space $\Lie(H)$. Therefore, if $\Delta=\zspan\{\exp^{-1}{\Lambda}\}$, then
 \begin{eqnarray*}
 [\Delta,\Delta]\subseteq[\exp^{-1}\Lambda,\exp^{-1}\Lambda]\subseteq\zspan\{{\exp^{-1}[\Lambda,\Lambda]}\}=\exp^{-1}{[\Lambda,\Lambda]}.
 \end{eqnarray*}
\end{proof}
\subsection{$(r,u_t,x)$-dominant subgroups}
 For the remainder of this section, let $\Gamma$ be a discrete group  of $G$, $T>0$, $0<\delta\leqslant r_0$,  $x\in X_{\delta}$, and  $\{u_t\}_{t\in\mathbb{R}}$ be a one-parameter  unipotent subgroup of $G$. In the following, an interval $I\subseteq\mathbb{R}$ means a connected subset of $\mathbb{R}$.
\begin{definition}
 Let $I$ be an interval in $\mathbb{R}$ and let $r>0$. We say that a subgroup $H$ of $G$ is an $(r,u_t,x)$-dominant group 
 on the interval $I$
 if $G_{u_tx}\cap B_{r}\subseteq u_tHu_{-t}$ for all $t\in I$. \end{definition}

Consider the intersection $G_{u_tx}\cap B_{\delta}$ for $t\in[0,T]$, and we can split up $[0,T]$ into subintervals such that
\begin{enumerate}
 \item $[0,T]=I_1\sqcup I_2\sqcup...\sqcup I_k$ (disjoint union)
\item $G_{u_tx}\cap B_{\delta}=\{e\}$ on $I_1,I_3,...$ (that is, for all $t$ in $I_1, I_3,...$) which are closed in $[0,T]$
\item $G_{u_tx}\cap B_{\delta}\neq \{e\}$ on $I_2,I_4,...$ which are open in $[0,T]$
\end{enumerate}
The following Lemma guarantees that our partition is finite.
\begin{lemma}\label{l36} Let $G$ be an arbitrary  Lie group, let $\|\cdot\|$ be a norm on $\Lie(G)$, and let $r_0$ be such that the natural log is defined on $B_{r_0}$. Let $\Gamma$ be a discrete subgroup of $G$, let $\{u_t\}_{t\in\mathbb{R}}$ be a one-parameter unipotent subgroup of $G$, and $x\in G/\Gamma$. Then for any $T>0$ and $0<\delta<r_0$ the number of connected components of $\{t\in[0,T]:G_{u_tx}\cap B_{\delta}\neq\{e\}\}$ is finite.
\end{lemma}
\begin{proof}Assume that $\{t\in[0,T]:G_{u_tx}\cap B_{\delta}\neq\{e\}\}$ has infinitely many connected components.  By the compactness of 
$[0,T]$ there exists a point $t\in [0,T]$ such that for every $\epsilon>0$ $(t-\epsilon,t+\epsilon)$ intersects infinitely many elements of the partition. Hence we may find points $s_1,s_2,\dots$ such that each $s_i$  is in a distinct connected component 
and $s_i\rightarrow t$. There exists $\gamma_i\in \Gamma$ such that $u_{s_i}\gamma_iu_{-s_i}\in B_{\delta}$. By discreteness of $\Gamma$, after passing to a subsequence, we may assume that $\gamma_i=\gamma_1$ and let $x=log(u_t\gamma_{1}u_{-t})\in\overline{\mathfrak b_\delta}$. But $Ad_{u_{s-t}}x$ can not enter and leave $\overline{\mathfrak b_\delta}$ infinitely often since the map $s\mapsto Ad_{u_{s-t}}x$ is a polynomial map. This contradicts the infinitely many components.

\end{proof}

\begin{lemma}\label{l37}
 For each interval $I_j$ on which $G_{u_tx}\cap B_{\delta}\neq\{e\}$, there exists a Lie subgroup $H$ of $G$
 which is a $(\delta,u_t,x)$-dominant group on $I_j$ and $H$ is conjugate to a subgroup of $MA$ or $MN$.
\end{lemma}
\begin{proof}
Choose $t_0\in I_j$. By Lemma \ref{l31} we can find a subgroup $F_0\subseteq G$ such that 
$G_{u_{t_0}x}\cap B_{\delta}\subseteq F_0$, and $F_0$ is conjugate to $MA$ or $MN$.  As $\Gamma$ is torsion free and $G_{u_{t_0}x}$ is conjugate to $\Gamma$, 
$F_0$ is not contained in a conjugate of $ M$ and our choice for $F_0$ is unique by Lemma \ref{l32}. Let $F=u_{-t_0}F_0u_{t_0}.$

Let$$a=\inf\left\{t\in I_j:G_{u_{s}x}\cap B_{\delta}\subseteq u_{s}F u_{s}^{-1}, s\in[t,t_0]\right\}.$$ We shall show that 
$I_j\subseteq [a,\infty)$. If not, then $a\in\mathbb{R}$ and there exists
 a subgroup $H$ conjugate to either $MA$ or $MN$ such that $G_{u_{a}x}\cap B_{\delta}\subseteq H$. We can choose 
 a small number $\epsilon>0$ such that
$u_{\epsilon}Hu_{\epsilon}^{-1}\cap G_{u_{a+\epsilon}}\cap B_{\delta}\neq \{e\}$. Then $u_{a+\epsilon}F u_{a+\epsilon}^{-1}$ 
and $u_{\epsilon}Hu_{\epsilon}^{-1}$
have nontrivial intersection. Since $\Gamma$ was torsion free, the intersection must fall outside of any conjugate of 
$M$ and hence  by Lemma \ref{l32}, 
$u_{\epsilon}Hu_{\epsilon}^{-1}= u_{a+\epsilon}Fu_{a+\epsilon}^{-1}$, and $H=u_{a}Fu_{a}^{-1}$.
We can choose a small number $\epsilon>0$ such that for any $c\in (-\epsilon,\epsilon)$, $u_{c}Hu_{c}^{-1}\cap G_{u_{a+c}
x}\cap B_{\delta}\neq \{e\}$. Again,  since $\Gamma$ was torsion free, the intersection must fall outside of any 
conjugate of $M$ and by Lemma \ref{l32}, $G_{u_{a+c}}\cap B_{\delta}\subseteq u_{c}Hu_{c}^{-1}=u_{a+c}F u_{a+c}^{-1}$. 
This contradicts the choice of $a$.

To finish the proof, repeat the above argument with $$b=\sup\left\{t\in I_j=(a_j,b_j)|G_{u_{s}x}\cap B_{\delta}\subseteq u_{s}
Fu_{s}^{-1}, s\in[t_0,t]\right\}$$ to get $I_j\subseteq (-\infty, b]$.
\end{proof}

\section{Non-Divergence:  $\mathbb{R}$-rank 1}\label{prim}
  We shall use the  non-divergence results of Margulis \cite{M},  Dani \cite{D} \cite{D2}, Kleinbock and Margulis \cite{KM}, and Kleinbock \cite{K}  for actions of one-parameter unipotent subgroups of $\SL(k,\mathbb{R})$ on unimodular  lattices in $\mathbb{R}^k$.
  \subsection{Primitive subgroups of discrete subgroups of $\mathbb{R}^k$}
 For a discrete subgroup $\Lambda$, let $
\mathcal{L}(\Lambda)$ denote the set of all primitive subgroups  contained in $\Lambda$. In other words,
$$\mathcal{L}(\Lambda):=\left\{\Delta\subseteq\Lambda :\text{ $\rspan$}(\Delta)\cap\Lambda=\Delta\right
\}.$$
 For a discrete subgroup $\Delta\subseteq\mathbb{R}^k$, let $\|\Delta\|_0$ denote the covolume of $\Delta$ 
 in its $\rspan$. 
 
 Fix an orthonormal basis $\{e_1,\dots,e_k\}$ of $\mathbb{R}^k$. For $0<m\leqslant k$, choose a Euclidean norm $\|\cdot\|_{m,k}$ on $\wedge^m\mathbb{R}^k$ such that 
 the set
 $$\{e_{i_1}\wedge\dots\wedge e_{i_m}:i_1<\dots<i_m\}$$
 is an orthonormal basis of $\wedge^m\mathbb{R}^k$. For any discrete subgroup $\Lambda$ in $\mathbb{R}^k$ of rank $m$, we have
 $$\|\Lambda\|_0=\|\gamma_1\wedge\dots\wedge\gamma_m\|_{m,k}$$
 where $\{\gamma_1,\dots,\gamma_m\}$ is a basis for $\Lambda$. Since the set $\{\gamma_1\wedge\dots\wedge\gamma_m:\gamma_1\dots\gamma_m\in\Lambda\}$ is discrete in $\wedge^m\mathbb{R}^k$, we have the following lemma.
 \begin{lemma}\label{l41} Let $\Lambda$ be a discrete subgroup of $\mathbb{R}^k$. For any $C>0$,
 $$\#\left\{\Delta\in\mathcal{L}(\Lambda):\|\Delta\|_0<C\right\}<\infty.$$
 \end{lemma}
\subsection{Quantitative non-divergence in the space of unimodular lattices}

 For positive numbers $C$ and $\alpha$ and a subset $B$ of $\mathbb{R}$, we say that a function $f: B\to\mathbb{R}$ is $(C,\alpha)$-good on $B$ if for any open interval $J\subseteq B$ and any $\epsilon>0$, $f$ satisfies the relation
 $$m\left(\left\{ t\in J:|f(x)|<\epsilon\right\}\right)\leqslant C\left(\frac{\epsilon}{\sup_{t\in J}|f(t)|}\right)^\alpha m(J).$$
The following two propositions identify some functions which are $(C,\alpha)$-good. For a more detailed explanation of which functions are $(C,\alpha)$-good, see \cite{K} and \cite{KM}.
\begin{proposition}[{\cite[Proposition 3.2]{K}\cite[Lemma 4.1]{DM}}]\label{p41} For any $k\in\mathbb{N}$, any polynomial $f\in\mathbb{R}[x]$ of degree not greater than $k$ is $\left(k(k+1)^\frac{1}{k},\frac{1}{k}\right)$-good on $\mathbb{R}$.
\end{proposition}
\begin{proposition}\label{p42} Let $f$ be a function on $\mathbb{R}$. If $f$ is $(C,\alpha)$-good on $\mathbb{R}$, then $\sqrt{f}$ is $(C,2\alpha)$-good on $\mathbb{R}$.
\end{proposition}
\begin{proof} The obvious calculation yields the result. 
\end{proof}

 \begin{proposition}[cf.{\cite[Corollary 3.3]{K}}]\label{p43} For any $k\in\mathbb{N}$, any one-parameter unipotent subgroup $\{u_t\}_{t\in\mathbb{R}}$ of $\SL(k,\mathbb{R})$,  any lattice $\Lambda$ in $\mathbb{R}^k$, and any subgroup $\Delta$ of $\Lambda$, the function $t\mapsto\|u_t\Delta\|_0$ is $(k^2(k^2+1)^\frac{1}{k^2},\frac{1}{k^2})$-good. \end{proposition}
\begin{proof}
Suppose the rank of $\Delta$ is $m$. Let $\{\gamma_1,\dots,\gamma_m\}$ be a generating set for $\Delta$. Then for all $t\geqslant0$,
$$\|u_t\Delta\|_0=\|(\wedge^m u_t)(\gamma_1\wedge\dots\wedge\gamma_m)\|_{m,k},$$
where $(\wedge^m u_t)$ acts by $u_t$ on each coordinate of $\gamma_1\wedge\dots\wedge\gamma_m$. The action of $(\wedge^m u_t)$ on each basis vector $e_{i_1}\wedge\dots\wedge e_{i_m}$ with $i_1<\dots<i_m$ of $\wedge^m\mathbb{R}^k$, is a polynomial of degree at most $\frac{k(k-1)}{2}$. Hence $\|u_t\Delta\|_0^2=\|(\wedge^m u_t)(\gamma_1\wedge\dots\wedge\gamma_m)\|_{m,k}^2$ is a polynomial of degree at most $\frac{k(k-1)}{2}$. The claim then follows from  Propositions \ref{p41} and \ref{p42}.
\end{proof}

\begin{proposition}\label{p44}
Let $\{u_t\}_{t\in\mathbb{R}}$ and $\Delta$ be as in Proposition \ref{p43}. Suppose there exists $r>0$ such that $\|u_t\Delta\|_0<r$ for all $t\geqslant0$. Then $\|u_t\Delta\|_0=\|\Delta\|_0<r$.
\end{proposition}
\begin{proof}
This follows from the fact that the map $t\mapsto\|u_t\Delta\|_0^2$ is a polynomial map. 
\end{proof}

  For a discrete subgroup $\Lambda$ in $\mathbb{R}^k$,  define 
  $$d_1(\Lambda)=\inf_{v\in\Lambda\setminus\{0\}}\|v\|.$$ 
  
  We have the following elementary result for discrete groups in $\mathbb{R}^k$ (see \cite{K} page 8 for a sharper result).
\begin{lemma}\label{l42} Let $\Delta$ be a discrete subgroup of $\mathbb{R}^n$. Then 
$$d_1(\Delta)\leqslant4(\|\Delta\|_0)^{\frac{1}{\rank(\Delta)}}.$$
\end{lemma}
 The following constitutes our main tool.

 \begin{theorem}[{\cite[Theorem 3.4]{K}}]\label{thm4main}
 Suppose an interval $B\subseteq\mathbb{R}$, $C,\alpha>0$, $0<\rho<1$,  and a continuous map $h:B\rightarrow \SL(k,\mathbb{R})$ are given.  Assume that for any $\Delta\in\mathcal{L}(\mathbb{Z}^k)$, 
 \begin{enumerate}
 \item The function $t\mapsto \|h(t)\Delta\|_0$ is $(C,\alpha)$-good on $B$, and
 \item $\sup_{t\in B}\|h(t)\Delta\|_0\geqslant\rho^{\rank(\Delta)}$.
 \end{enumerate}
 Then for any $0<\epsilon<\rho$,
 \begin{equation}
m\left(\left\{t\in B : d_1(h(t)\mathbb{Z}^k)<\epsilon\right\}\right)\leqslant k2^kC\left(\frac{\epsilon}{\rho}\right)^\alpha m(B).
 \end{equation}
 \end{theorem}

We combine the above results in the following.
 
\begin{corollary}[\cite{KM},\cite{K}]\label{thm4} Let 
 $\Lambda$ be a discrete subgroup of $\mathbb{R}^k$, $B$ be an interval in $\mathbb{R}$, and let
 \begin{align*}\rho=\rho(\Lambda,B)&:=\min\left(1/k,\inf_{\Delta\in\mathcal{L}(\Lambda)}\left\{\sup_{t\in B}\|h_t\Delta\|_0^\frac{1}{\rank(\Delta)}\right\}\right)\\
 &\geqslant\min\left(1/k,\frac 1 4\sup_{t\in B}\;d_1(h_t\Lambda)\right),\\
C_k&:=k^32^{k}(k^2+1)^{1/k^2},\\ \alpha_k&:=1/k^2.\end{align*}
 Then for
 any one-parameter unipotent subgroup $\{h_t\}_{t\in\mathbb{R}}$ of $\SL(k,\mathbb{R})$
 and any $0<\epsilon<\rho$, one has
 $$m\left(\left\{t\in B:d_1(h_t\Lambda)<\epsilon\right\}\right)\leqslant C_k\left(\frac{\epsilon}{\rho}
 \right)^{\alpha_k}m(B).$$
\end{corollary}
\begin{proof}By Theorem 3.9 of \cite{K}, one may replace $\mathbb{Z}^k$ in Theorem \ref{thm4main} with any discrete subgroup of  $\mathbb{R}^k$. By Proposition \ref{p43}, for every $\Delta\in\mathcal{L}(\Lambda)$ the function $\|h(t)\Delta\|_0$ is $(k^2(k^2+1)^\frac{1}{k^2},\frac{1}{k^2})$-good on $[0,T]$. Thus the first condition of Theorem \ref{thm4main} is satisfied, and by our choice of $\rho$ the second is also satisfied.\\

Note that by Lemma \ref{l42}, 
\begin{equation*}
\inf_{\Delta\in\mathcal{L}(\Lambda)}\left\{\sup_{t\in B}\|h_t\Delta\|_0^\frac{1}{\rank(\Delta)}\right\}
 \geqslant\inf_{\Delta\in\mathcal{L}(\Lambda)}\left\{\frac 1 4\sup_{t\in B}\;d_1(h_t\Delta)\right\}
 \geqslant\frac 1 4\sup_{t\in B}\;d_1(h_t\Lambda).
 \end{equation*}
 This gives the inequality in the condition of $\rho$.
\end{proof}

\subsection{Proof of Theorem \ref{maintheorem1}}
\begin{proof} Let $\Gamma$ be a virtually torsion free discrete subgroup of $G$. By obvious adjustments to $c_\epsilon$, without loss of generality, we may assume that $\Gamma$ is torsion free.   Let  $\epsilon>0$ and $0<\delta\leqslant r_0$. Fix    a one-parameter  unipotent subgroup $\{u_t\}_{t\in\mathbb{R}}$  of $G$, $x\in X_{\delta}$, and $T>0$. Then there exist finitely many disjoint intervals $I_1,I_2,I_3\dots$ 
such that \begin{enumerate}
\item $[0,T]=I_1\cup I_2\cup I_3\cup...$
\item $G_{u_tx}\cap B_{\delta}=\{e\}$ on $I_1,I_3,...$ which are closed in $[0,T]$
\item $G_{u_tx}\cap B_{\delta}\neq \{e\}$ on $I_2,I_4,...$ which are open in $[0,T]$.
\end{enumerate}
By the fact that the intervals are disjoint and in view of (\ref{thm3}), it is enough to show that there exists $
c_\epsilon>0$ independent of $x$, $\{u_t\}_{t\in\mathbb{R}}$, and $T $ such that for each $j\in2\mathbb{Z}$,

\begin{equation}\label{eqnth411}
  m\left(\left\{t\in I_j:G_{u_tx}\cap B_{c_\epsilon\delta}\neq \{e\}\right\}\right)\leqslant\epsilon\cdot m(I_j).
 \end{equation}
Let $j$ be even and $I_j=(a_j,b_j)$. By Lemma \ref{l37}  there 
exists  an $(\delta,u_t,x)$-dominant subgroup $H_j$ on $I_j$ such that $H_j$ is conjugate to a subgroup of $MN$ or  $MA$.\\

Let $F_j=G_{x}\cap H$.  By  Corollary \ref{c33}  and 
Lemma \ref{l34}, there exists  a subgroup $\tilde{F_j}$ of $F_j$ and a constant $l_M$ depending only on $M$ such that $[F_j:\tilde{F_j}]\leqslant l_M$ and $\tilde{F_j}$ is contained in a 
connected 2-step nilpotent subgroup $\tilde{H_j}$ of $H_j$.\\

By Lemma \ref{l35}, $\Lambda_j:=\zspan\{\exp^{-1}(\tilde{F})\}$
is a discrete subgroup in $\Lie(\tilde{H_j})\subseteq\Lie(G)$. We apply the Corollary \ref{thm4} to $\Lambda_j$ in the case of $B=I_j$ and $h_t=\Ad_{u_{t}}$. Since $u_t$ is a one-parameter unipotent  subgroup, $h_t=\Ad_{u_t}$ is a one-parameter unipotent subgroup of $\SL(k,\mathbb{R})$, where $k$ is the dimension of $\Lie
(G)$.
By Corollary \ref{thm4}, for any $0<\eta\leqslant\rho_j=\rho(\Lambda_j, I_j)$,
\begin{equation}\label{eqnth412}m\left(\left\{x\in I_j:d_1(h_x\Lambda_j)<\eta\right\}\right)\leqslant C_k\left(\frac{\eta}{\rho_j}\right)^
{\alpha_k}m(I_j).\end{equation}\\

Further assume that $0<\eta\leqslant\min\{r_0,\rho_j\}$.
For every $f\in \tilde{F_j}$ and $t\in I_j$, $\Vert u_t f u_{-t}\Vert\leqslant\eta$ implies that $d_1
(h_t\Lambda_j)<\eta$. Thus 
$$ m\left(\left\{t\in I_j:u_t\tilde{F_j} u_{-t}\cap B_{\eta}\neq\{e\}\right\}\right)
\leqslant m\left(\left\{t\in I_j:d_1(h_t\Lambda_j)<\eta\right\}\right).$$
For any $g\in F_j$, $g^{l_M}\in \tilde F_j$. Hence if $\|g\|<\eta/l_M$, then $\|g^{l_M}\|<\eta$. Thus
\begin{eqnarray*}
m(\{t\in I_j:G_{u_tx}\cap B_{\eta/l_M}\neq\{e\}\})
&\leqslant&  m\left(\left\{t\in I_j:u_t\tilde{F} u_{-t}\cap B_{\eta}\neq\{e\}\right\}\right)\\
&\leqslant& m\left(\left\{t\in I_j:d_1(h_x\Lambda)<\eta\right\}\right)\\
&\leqslant& C_k\left(\frac{\eta}{\rho_j}\right)^{\alpha_k}m(I_j).
\end{eqnarray*}\\
 By construction of $F_{j}$, we have $u_{a_{j}}F_{j}u_{-a_{j}}\cap B_{\delta}=\{e\}$. Thus, by Lemma \ref{l35}, $\Ad_{u_{a_{j}}}(\Lambda_{j})\cap \mathfrak{b}_{\frac{\delta}{2}}=\{0\}$.  Let $\Delta\in\mathcal{L}(\Lambda_{j})$. Then by  Lemma \ref{l42},
\begin{equation*}
\sup_{t\in B}\|h_t\Delta\|_0^\frac{1}{\rank(\Delta)}\geqslant\|h_{a_j}\Delta\|_0^\frac{1}{\rank(\Delta)}\geqslant\frac{d_1(\Delta)}{2}\geqslant\frac{\delta}{8}.
\end{equation*}

Recall that we chose $r_0<\frac{1}{k}$, so $\rho\geqslant\frac{\delta}{8}$. Choose $\eta $ such that ${C_k}\left(\frac{8\eta}{\delta}\right)^{{\alpha_k}}=\epsilon$ and let $\delta=\eta/l_M$.  By above, (\ref{eqnth411}) holds for every interval $I_j$ with 
\begin{equation*}
c_\epsilon=\left(\frac{\epsilon}{C_k}\right)^{k^2}\frac{1}{8l_M}. \tag{A}
\end{equation*}
\end{proof}

\subsection{Proof of Theorem \ref{uniform}}
\begin{proof}
Let $\epsilon>0$ be given. Fix $r_0>0$ such that $B_{r_0}$ is a Zassenhaus neighborhood and $\eta\in(0,1]$. By the proof of Theorem \ref{maintheorem1}, for $\delta=c_{\epsilon/2}\eta r_0$,  if $y \in X_{\eta r_0}$, then for any unipotent one parameter subgroup $\{u_t\}_{t\in\mathbb{R}}$ of $G$,
\begin{equation*}
\frac{1}{T}m\left(\left\{t\in[0,T]: u_ty\not\in X_{\delta}\right\}\right)<\frac12\epsilon, \,\,\text{for all $T>0$}.
\end{equation*}
Therefore, if $x\in X$  and $t_0\geqslant0$ are such that $u_{t_0}x\in X_{\eta r_0}$, then for $T> t_0(2\epsilon^{-1}-1)$
\begin{equation*}
\frac{1}{T}m\left(\left\{t\in[0,T]: u_tx\not\in X_{\delta}\right\}\right)<\epsilon.
\end{equation*}
We will prove the theorem with
 \begin{equation*}
 \delta_\epsilon=c_{\epsilon/2} r_0\quad \text{and} \quad \delta=\eta\delta_\epsilon,\tag{B}
 \end{equation*}
where $c_{\epsilon}$ is as in Theorem \ref{maintheorem1}. By above we may assume that $x=g\Gamma\in X$ is  such that $u_tx\not\in X_{\eta r_0}$ for all $t\geqslant0$. By Lemma \ref{l37}, there exists an $\left(\eta r_0,u_t,x\right)$-dominant subgroup $H$ on $[0,\infty)$ such that $H$ is conjugate to $MA$ or $MN$. Let $F=g\Gamma g^{-1}\cap H\neq\{e\}$ which is torsion-free.  By  Corollary \ref{c33}  and 
Lemma \ref{l34}, there exists  a subgroup $\tilde{F}$ of $F$ and a constant $l_M$ depending only on $M$  such that $[F:\tilde{F}]\leqslant l_M$ and $\tilde{F}$ is contained in a 
connected 2-step nilpotent subgroup $\tilde{H}$ of $H$.\\

Define $\Lambda_0:=\zspan\{\exp^{-1}(\tilde{F})\}$. Again, by Lemma \ref{l35}, $\Lambda_0$
is a discrete subgroup in $\Lie(\tilde{H})$.  We proceed now as in the proof of Dani's uniform non-divergence result \cite[Theorem 2.1]{D2}. 
Suppose there exists $T_0\geqslant0$ such that
\begin{equation}\label{rhocrit}
\sup_{t\in[0,T_0]}\left\{\|\Ad_{u_t}\Lambda\|_0^\frac{1}{\rank(\Lambda)}\right\}\geqslant\frac{\eta r_0}{8}\text{ for all  $\Lambda\in \mathcal{L}(\Lambda_0)$}.
\end{equation}
Again recall that $r_0<\frac{1}{k}$, so
$$\rho\big(\Lambda_0,[0,T_0]\big)\geqslant \frac{\eta r_0}{8}.$$
By Corollary \ref{thm4} and in view of (A) and (B),  for all $T>T_0$,
\begin{equation*} m\left(\left\{t\in [0,T]:d_1(\Ad_{u_t}\Lambda_0)<l\delta\right\}\right)\leqslant C_k\left(\frac{8l\delta}{\eta r_0}
 \right)^{\alpha_k}T=T\frac{\epsilon}{2}.
 \end{equation*}
Thus
\begin{equation*}
\frac{1}{T}m\left(\left\{t\in[0,T]: u_tx\not\in X_\delta\right\}\right)<\frac12\epsilon, \,\,\text{for all $T>T_0$}.
\end{equation*}
In this case we are done.\\

 Otherwise, (\ref{rhocrit}) fails for every $T_0\geqslant 0$. By Lemma \ref{l41},  the set $$\mathcal{F}=\left\{\Lambda\in \mathcal{L}(\Lambda_0):\|\Lambda\|_0^\frac{1}{\rank(\Lambda)}<\frac{\eta r_0}{8}\right\}$$ is finite. Therefore, as (\ref{rhocrit}) fails for all $T_0\geqslant 0$, there exists a $\Lambda\in\mathcal{F}$ such that $\|\Ad_{u_t}\Lambda\|_0<\left(\frac{\eta r_0}{8}\right)^{\rank(\Lambda)}$ for all $t\geqslant0$. By Proposition \ref{p44}, it follows that 
 $\|\Ad_{u_t}\Lambda\|_0=\|\Lambda\|_0<\left(\frac{\eta r_0}{8}\right)^{\rank(\Lambda)}$.
Thus the space $W:=\rspan(\Lambda)$ is invariant under $\Ad_{u_t}$, $\Lambda$ is a lattice in $W$, and $\|\Ad_{u_t}\Lambda\|_0$ is a constant less than $\left(\frac{\eta r_0}{8}\right)^{\rank(\Lambda)}$ for all $t>0$. \\

We want to prove that there exists an abelian subgroup $L$ such that $L$ is $u_t$-invariant and $L\cap{F}$ is a lattice of $L$ with small covolume. Note that if $N$ is abelian, by Lemma \ref{l34}, the subspace $W$ constructed above is an abelian subalgebra and we can take $L=\exp(W)$. But in general, we need more analysis.\\

Recall $W\subseteq \Lie(\tilde{H})$. As $W$ is $\Ad_{u_t}$-invariant, $\exp(\Lambda)\subseteq u_t\tilde{H}u_{-t}\cap\tilde{H}$. Hence $u_t\tilde{H}u_{-t}\cap\tilde{H}$ contains a  nontrivial non-elliptic element for all $t>0$. Therefore, by Corollary \ref{c32}, for all $t>0$ $u_t{H}u_{-t}= {H}$ and $\{u_t\}\subset {H_U}$, where $H_U$ is the unipotent radical of $H$ which is conjugate to $N$. Thus $\Ad_{u_t}$ acts trivially on $[\Lie(\tilde{H}),\Lie(\tilde{H})]$.

 Define 
 $$S=\left\{\lambda\in\Lambda: \exists\,\,  t>0\,\,\text{such that}\,\,\|u_t\lambda\|<\frac{\eta r_0}{2}\right\}.$$
 Fix $\lambda=v+c$ in $S$, where $c\in[\Lie(H_U),\Lie(H_U)]$ and $v\in\Lie(\tilde{H})$ is in the orthocomplement to $[\Lie(H_U),\Lie(H_U)]$.  As $\tilde{H}$ is 2-step and $u_t\in H_U$, 
 $$\Ad_{u_t}(\lambda)=v+t[v,u]+c,$$ where $u_t=\exp(tu)$ and $u\in \Lie(H_U)$.
 By definition of $S$,  there exists $t>0$ such that $\|\Ad_{u_t}\lambda\|<\frac{\eta r_0}{2}$. Thus by Lemma \ref{l34}, $[v,u_t]\in [\Lie(H_U),\Lie(H_U)]$, $v$ is orthogonal to $[v,u_t]+c$, and
  \begin{equation}\label{2step}\|v\|<\frac{\eta r_0}{2}.\end{equation} 
 Assume  the Lie algebra $V:=\rspan\{[S,S]\}\neq\{0\}$.  Let $\gamma_1,\gamma_2\in S$. For $i=1,2$ there exists $c_i\in[\Lie(H_U),\Lie(H_U)] $ and $v_i$ in the orthogonal compliment to $[\Lie(H_U),\Lie(H_U)]$ such that $\gamma_i=v_i+c_i$. By (\ref{2step}), $\|v_1\|,\|
 v_2\|<\frac{\eta r_0}{2}$, and by the fact $\tilde{H}$ is 2-step, $[\gamma_1,\gamma_2]=[v_1,v_2]$. Thus $$\|[\gamma_1,\gamma_2]\|<(\frac{\eta r_0}{2})^2.$$ Let $\Delta=\zspan\{S\}$.  Then
$[\Delta,\Delta]$ is generated by elements whose norm is bounded by $(\frac{\eta r_0}{2})^2$. By Lemma \ref{l35}, $[\Delta,\Delta]$ is discrete. Thus the covolume of $[\Delta,\Delta]$ in $V$ is at 
 most $(\frac{\eta r_0}{2})^{2\dim(V)}$. As $V$ is contained in the commutator, $V$ is $\Ad_{u_t}$ invariant, and 
 $$\|\Ad_{u_t}\Delta\|_0=C<\left(\frac{\eta r_0}{2}\right)^{2\dim(V)}\text{ for all 
 $t>0$.}$$ By Lemma \ref{l35},  $\Lambda_0\subseteq \frac12 \exp^{-1}(\tilde{F})$. Let $L=\exp(V)$. From above, $L$ is abelian, 
 $L\cap{F}$ is a lattice in $L$, $L$ is $u_t$-invariant, and the covolume of $u_t(L\cap\tilde{F})u_{-t}$  in
  $u_tLu_{-t}=L$ with respect to the measure induced by our right invariant inner-product is a constant which is  at most 
 $( \frac{\eta r_0}{l})^{2\dim(L)}=( \frac{1}{lc_{\epsilon/2}}\delta)^{2\dim(L)}$.

Now assume that $V=\{0\}$. 
As shown previously, $\|\Ad_{u_t}\Lambda\|_0<(\frac{\eta r_0}{8})^{\rank(\Lambda)}$. Thus, by Lemma \ref{l42}, for every $t>0$,  $d_1(\Ad_{u_t}\Lambda)<\frac{\eta r_0}{2}$. As before, let $\Delta=\zspan\{S\}$. By the definition of $S$, for every $t>0$,
\begin{equation}\label{deltadiv}
d_1(\Ad_{u_t}\Delta)<\frac{\eta r_0}{2}.
\end{equation}
We claim that (\ref{deltadiv}) forces $\rho(\Delta,[0,T])\leqslant C_k^{\frac{1}{\alpha_k}}\eta r_0$ for all $T>0$. For the sake of a contradiction, assume that there exists a $T_0>0$ such that $\rho(\Delta,[0,T_0])>C_k^{\frac{1}{\alpha_k}}\eta r_0$. Then $\frac{\eta r_0}{2}<C_k^{\frac{1}{\alpha_k}}\eta r_0$ and by Corollary \ref{thm4},
\begin{equation}\label{falsediv}
m\left(\left\{t\in[0,T_0]:d_1\left(\Ad_{u_t}\Delta\right)<\frac{\eta r_0}{2}\right\}\right)\leqslant T_0 C_k \left(\frac{\frac{\eta r_0}{2}}{C_k^{\frac{1}{\alpha_k}}\eta r_0}\right)^{\alpha_k}=\frac{T_0}{2^{\alpha_k}}<T_0.
\end{equation}
Now (\ref{deltadiv}) and (\ref{falsediv}) contradict each other. Hence for all $T>0$, we have  $\rho(\Delta,[0,T])\leqslant C_k^{\frac{1}{\alpha_k}}\eta r_0$. By the definition of $\rho$ and Lemma \ref{l41}, repeat the argument above, there exists $\Delta_1\subset \Delta$ such that $\|\Ad_{u_t}\Delta_1\|_0\leqslant C_k^{\frac{1}{\alpha_k}}\eta r_0$ for all $t>0$. By Proposition \ref{p44}, if we take $W_1:=\rspan\{\Delta_1\}$, then $W_1$ is $\Ad_{u_t}$ invariant, $\Delta_1$ is a lattice in $W_1$, and $\|\Ad_{u_t}\Delta_1\|_0$ is a constant less than $(C_k^{\frac{1}{\alpha_k}}\eta r_0)^{\rank(\Delta_1)}$.

As $V=\{0\}$, $ W_1$ is a Lie algebra.  In this case, let $L=\exp(W_{1})$. Thus $L\cap
\tilde{F}$ is a lattice in $L$, $L$ is $u_t$-invariant, and the covolume of $u_t(L\cap\tilde{F})u_{-t}$  in 
$u_tLu_{-t}$ with respect to the measure induced by our right invariant inner-product is a constant which is at most $\left(2C_k^{\frac{1}{\alpha_k}}\frac{\eta r_0}{l}\right)^{\dim(L)}=\left(\frac{2C_k^{\frac{1}{\alpha_k}}}{lc_{\epsilon/2}}\delta\right)^{\dim(L)}$.   Thus the theorem holds with $\beta_\epsilon=\max\left\{\frac{2^{k^2+4}C_k^{2k^2}}{\epsilon^{k^2}},\frac{2^{k^2+3}C_k^{2k^2}}{\epsilon^{2k^2}}\right\}$.
\end{proof}

\section{Non-Divergence: Products of $\mathbb{R}$-rank 1 groups}

The case when $G$ is the product of semisimple groups of $\mathbb{R}$-rank $1$ is more delicate.  We will again partition $[0,T]$ into intervals where the intersection $G_{u_tx}\cap B_
{r_0}$ is trivial and intervals where the intersection is nontrivial. However, we will not be able, as in the case of 
Theorem \ref{maintheorem1},  to find a dominant nilpotent subgroup on each interval with nontrivial 
intersection. Indeed, there exist intervals  with nontrivial intersection, along which no dominant nilpotent 
subgroup is possible.\\

 To handle this problem, we break each interval down further and find dominant nilpotent 
subgroups on each sub interval. Yet this alone will not complete the proof, as we must also guarantee that on each 
sub interval the covolume  of the lattice in the Lie algebra of our dominant subgroup and all of its sub lattices 
obtain a fixed positive lower bound. To ensure a time with large covolume we extend each subinterval so that the enlarged interval includes such a time.  This extension allows us to apply the theorem of Kleinbock and Margulis, but in doing so our intervals are no longer disjoint. Fortunately, by our restriction  on the discrete subgroup,  the potentially bad times occurring on the intersection of two extended intervals can be measured by  studying a discrete group in a connected nilpotent subgroup having projections on less coordinates. The theorem then follows from induction on the number of projections on nontrivial coordinates.\\
\subsection{Algebraic properties of products of $\mathbb{R}$-rank 1 semisimple Lie Groups}

In this section, 
let $G=G_1\times\dots\times G_n$ where each $G_i$ is a rank-$1$ semisimple Lie group, $\Gamma$ be a discrete group,  and $X=G/\Gamma$. Denote the projection of $G$ onto the $i$-th factor  by $
 \pi_i$.
We further assume $\Gamma$ satisfies the following property:
\begin{center}
$(*)$ for any element $\gamma\in\Gamma\setminus\{e\}$,  $\pi_i(\gamma)$ is not elliptic for $i=1,\dots,n$.
\end{center}

For each $G_i$, let $\mathfrak{g}_i$ be the Lie algebra of $G_i$. Analogously to section 2, define  the groups  
$A_i$, $M_i$ and $N_i$ in $G_i$ and let $\|\cdot\|_i$ be the norm derived from the Killing form on the Lie algebra $
\mathfrak{g}_i$. Let $\|\cdot\|=\sup_i\|\cdot\|_i$ be the norm on $\mathfrak{g}:=\Lie(G)$. Again, for sufficiently 
small $r$, $B_r\subseteq G$ will denote the exponential of the open ball of radius $r$ centered at the origin in $\mathfrak{g}$. Fix $r_0<\frac 1 k$ where $k$ is the dimension of $G$, such that 
$B_{r_0}$ is a Zassenhaus neighborhood for $G$ and  $\exp^{-1}$ is well defined on $B_{r_0}$.  As before, for $0<\eta\leqslant r_0$, let $X_\eta=\{x\in X:G_x\cap B_\eta=\{e\}\}$, and for $\eta>r_0$, let $X_\eta=X$.\\

 The following three lemmas generalize the corresponding lemmas from section 3 to the product case.

\begin{lemma}\label{l51}
 For any discrete subgroup $\Gamma$ satisfying $(*)$, $\Gamma\cap B_{r_0}$ is contained in a conjugate of 
 $H_1\times...\times H_n$ where each $H_i\in\{\{e\}, M_iA_i, M_iN_i\}$.
\end{lemma}
\begin{proof}
 By the definition of the Zassenhaus neighborhood, the subgroup
 $\Delta$ generated by $\Gamma\cap B_{r_0}$ is nilpotent. For $i=1,\dots,n$. $\pi_i(\Delta)$ is nilpotent, and since $\Gamma$ satisfies $(*)$, $\pi_i(\Delta)$ contains no nontrivial elliptic elements.
 If $\pi_i(\Delta)=\{e\}$, then set $H_i=\{e\}$. Otherwise, by Lemma \ref{l31},  $\pi_i(\Delta)$ is 
contained in either a conjugate of $M_iA_i$ or $M_iN_i$.\\
\end{proof}

\begin{lemma}\label{l52} 
 For any discrete subgroup $\Lambda$ in $F_1\times...\times F_n$ where each $F_i\in\{\{e\}, M_iA_i, M_iN_i\}$,
 there exists a subgroup $\Lambda_0$ of finite index such that
 $\Lambda_0$ is contained in a connected nilpotent subgroup of $F_1\times...\times F_n$. 
 Moreover, the index $[\Lambda,\Lambda_0]$ is bounded by a number $l_M$ depending only on $M=M_1\times 
\dots\times M_n$.
\end{lemma}
\begin{proof} It is straightforward to deduce this from Lemma \ref{l33}.\\
\end{proof}

\begin{lemma}\label{l53}
 Any connected nilpotent subgroup of $H_1\times...\times H_n$, where  each $H_i\in\{\{e\}, M_iA_i, M_iN_i\}$, is 
 2-step nilpotent.
\end{lemma}
\begin{proof}
 This follows from Lemma \ref{l34} via the projections $\pi_i$ to $H_i$.\\
\end{proof}

\subsection{Intervals of maximal intersection}
Fix a one-parameter unipotent subgroup $\{u_t\}_{t\in\mathbb{R}}$ in $G$.
\begin{definition} Let $x\in X$, $s>0$, and $0<r\leqslant r_0$.  Let $H$ be a subgroup of $G$. If  $G_{u_{s}x}\cap u_{s}Hu_{-s}\cap B_r\neq\{e\}$,
let 
$$a=\sup_{t<s}\left\{t:G_{u_tx}\cap u_tHu_{-t}\cap B_r=\{e\}\right\},$$
and
$$b=\inf_{t>s}\left\{t:G_{u_tx}\cap u_tHu_{-t}\cap B_r=\{e\}\right\}.$$
Otherwise, let $a=b=s$.
We define $I(H,s,r,x):=(a,b)$. For $T>0$ define 
$$I_T(H,s,r,x):=[0,T]\cap I(H,s,r,x).$$
When the choice of $x$ is clear, we will abbreviate $I(H,s,r,x)$ and $I_T(H,s,r,x)$) by $I(H,s,r)$  and $I_T(H,s,r)$ respectfully.
\end{definition}

{\flushleft \bf{Note:}} If $t\in I_T(H,s,r,x)$, then
\begin{equation}\label{same}
I_T(H,t,r,x)=I_T(H,s,r,x).
\end{equation}

\begin{definition} Let $H$ be a subgroup of $G$. We say that $H$ is a \PN-subgroup if  $\pi_j(H)$ is conjugate to $\{e\}$, $M_jA_j$, or $M_jN_j$ for $j=1,\dots n$. Let $\CPN$ denote the collection of all \PN-subgroups. Further,
 for $i\in\mathbb{Z}$, we let
 $$\CPN_i=\left\{H\in\CPN: \#\big\{j:\pi_j(H)\neq\{e\}\big\}=i\right\}$$
\end{definition}

The following proposition provides the main ingredient in the proof of Theorem \ref{maintheorem2} and follows from a result of Kleinbock \cite{K}. (See Theorem \ref{thm4main})

\begin{proposition}\label{p51} Let $0<r\leqslant r_0$.  Let $H\in\CPN$, and $x\in X$. Suppose there exists $T_x>0 $ such that $ G_{u_{T_x}x}\cap B_r=\{e\}$.  Then for any $T>T_x$ and $s\in[0,T]$,
\begin{equation*}
  m\Big(\big\{t\in I_T(H,s,r):G_{u_tx}\cap u_tHu_{-t}\cap B_{\overline{c}_\epsilon r}\neq \{e\}\big\}\Big)\leqslant\epsilon\cdot m( I_T(H,s,r)),
 \end{equation*}
 where $\overline{c}_\epsilon=\left(\frac{\epsilon}{C_k}\right)^{k^2}\frac{1}{8l_M}$  where $k$ is the dimension of $G$, $C_k=k^32^k(k^2+1)^\frac1{k^2}$, and $l_M$ is the constant as in Lemma \ref{l52}.

\end{proposition}
\begin{proof}  If $G_{u_{s}x}\cap u_{s}Hu_{-s}\cap B_r=\{e\}$, then by  definition, $I_T(H,s,r)=\emptyset$ and   the proposition holds for the $\overline{c}_\epsilon$. 

Now suppose that $G_{u_{s}x}\cap u_{s}Hu_{-s}\cap B_r\neq\{e\}$.
Let $F=G_x\cap H$. Then for all $t\in\mathbb{R}$,
 \begin{equation}\label{eqn1}
 G_{u_tx}\cap u_tHu_{-t}
 =u_{t}Fu_{-t}
 \end{equation}
 By Lemma \ref{l51}, Lemma \ref{l52}, and Lemma \ref{l53} there 
exists a subgroup $\tilde{F}\subseteq F$ such that $\tilde{F}$ is contained in a connected 2-step nilpotent subgroup $\tilde
{H}\subseteq H$ and $[F:\tilde{F}]\leqslant l_M$. Let $\Lambda=\zspan(\exp^{-1}\tilde{F})$.
By the assumptions that $ G_{u_{T_H}x}\cap B_r=\{e\}$ and $T>T_H$, there exists $a\in\overline{I_T(H,s,r)}$ such that   $\Ad_{u_a}\tilde{F}\cap B_{r}=\{0\}$. So by Lemma \ref{l35}, $d_1(\Ad_{u_a}(\Lambda))>\frac{r}2$. Then by Corollary \ref{thm4} with $\rho>r/8$ we have
\begin{eqnarray*}
  &{}&m\Big(\big\{t\in I_T(H,s,r))|u_tFu_{-t}\cap B_{\overline{c}_\epsilon r}\neq \{e\}\big\}\Big)\\
  &\leqslant&m\Big(\big\{t\in I_T(H,s,r))|u_t\tilde{F}u_{-t}\cap B_{l_M\overline{c}_\epsilon r}\neq \{e\}\big\}\Big)\\
  &\leqslant&m\Big(\big\{t\in I_T(H,s,r))|\Ad_{u_t}(\Lambda)\cap \mathfrak b_{l_M\overline{c}_\epsilon r}\neq \{e\}\big\}\Big)\\
  &\leqslant&m\Big(\big\{t\in I_T(H,s,r))| d_1(\Ad_{u_t}(\Lambda))< l_M\overline{c}_\epsilon r\big\}\Big)\\  
  &\leqslant& C_k\left(\frac{l_M\overline{c}_\epsilon r}{\rho}\right)^{\alpha_k}m(I_T(H,s,r)))\\ 
  &\leqslant&\epsilon\cdot m(I_T(H,s,r))).
 \end{eqnarray*}
The proposition then follows from (\ref{eqn1}).\\
\end{proof}




\begin{definition}Let $p\in\{0,1,\dots,n\}$ and $T>0$. Let $\mathcal{F}\subseteq \CPN_p\times [0,T]$. For $0<r<r_0$, let
$$I_T(\mathcal{F},r):=\bigcup_{(H,t)\in \mathcal{F}} I_T(H,t,r).$$
Note that if $p=0$, then $I_T(\mathcal{F},r)=\emptyset$.
\end{definition}

\begin{definition} Let  $H$ be a subgroup of $G$. Let $F_i=\{e\}$ if $\pi_i(H)=\{e\}$, otherwise let $F_i=G_i$. Define $\hull(H)=F_1\times\dots\times F_n$. \end{definition}

The following proposition is a crucial ingredient in our proof of the product case.

\begin{proposition}\label{p52} Fix $x\in X$. Let $H^1,H^2\in\CPN_p$ be such that $\hull(H^1)=
\hull(H^2) $ and $H^1\neq H^2$. Let  $T>0$, $0<r\leqslant r_0$, and $s_1,s_2\in[0,T]$.
 Then there exists a finite set $\mathcal{F}\subseteq\CPN_{p-1}\times [0,T]$ such that  
$$I_T(H^1,s_1,r)\cap I_T(H^2,s_2,r)\subseteq I_T(\mathcal{F},r) .$$
\end{proposition}
\begin{proof} We may assume  that  $I_T(H^1,s_1,r)\cap I_T(H^2,s_2,r)\neq\emptyset$, since otherwise the statement 
holds trivially.  For $i=1,2$, by the fact that $H^i\in{\CPN}_p$, we may assume without loss of generality, that
$$(g^{(i)})^{-1}H^ig^{(i)}= L_1^i\times\dots\times L_p^i\times\{e\}\times\dots
\times\{e\}$$ for some $g^{(i)}=g_1^{(i)}\times\dots\times g^{(i)}_n$ in $G$ and  $L_j^i\in\{M_jA_j,M_jN_j\}$ for $j=1,\dots,p$.
Since $H^1$ and $H^2$ are distinct we may assume that  $g_p^{(1)}L_p^1(g_p^{(1)})^{-1}\neq g_p^{(2)}L_p^2(g_p^
{(2)})^{-1}$. \\

Fix $t\in I_T(H^1,s_1,r)\cap I_T(H^2,s_2,r)$. For $i=1,2$, there exists nontrivial
$h^{(i)}\in H^i$ such that $(g^{(i)})^{-1}h^{(i)}g^{(i)}=l_1^{(i)} \times\dots\times l_p^{(i)}\times e\times\dots\times e$ with $l_j^{(i)}\in 
L_j^i$, for $j=1,\dots,p$ and 
$$u_th^{(i)}u_{-t}\in G_{u_tx}\cap u_tH^iu_{-t}\cap B_{r}.$$
  Since $B_{r}$ is a Zassenhaus  neighborhood, there exists a connected nilpotent group $H(t)$ such that  
  $$G_{u_tx}\cap B_{r}\subseteq H(t).$$
By Lemma \ref{l51},
  $H(t)$ is contained in   $H_1(t)\times\dots\times H_n(t)$ where each $H_i(t)$ is conjugate to $\{e\}$, $M_jA_j$, or $M_jN_j$.\\
  
Thus conjugates of $l_p^{(1)},l_p^{(2)}$ lie in $H_p(t)$. 
If $l_p^{(i)}$ is nontrivial, since $h^{(i)}\in\Lambda$, $\pi_p(h^{(i)})\in g_p^{(i)}L_p^i(g_p^{(i)})^{-1}\cap H_p(t)$ is non-elliptical, and hence by Corollary \ref{c32},  
$g_p^{(i)}L_p^i(g_p^{(i)})^{-1}= H_p(t)$. By the fact that $g_p^{(1)}L_p^1(g_p^{(1)})^{-1}\neq g_p^{(2)}L_p^2(g_p^
{(2)})^{-1}$,  it follows that $l_p^{(1)}$ and $l_p^{(2)}$ cannot both be nontrivial. For $i=1,2$ let
$$F^i=g^{(i)}\big(L_1^i \times\dots\times L_{p-1}^i\times\{e\}\times\dots\times\{e\}\big)(g^{(i)})^{-1}\in\CPN_{p-1}.$$

Without loss of generality assume that $l_p^{(1)}$ is trivial. Define  $F(t)=F^1$. By construction $h^{(1)}\in F(t)$ and 
$$u_th^{(1)}u_{-t}\in G_{u_tx}\cap u_tF(t)u_{-t}\cap B_r, \,\,\text{and $t\in I_T(F(t),t,r)$.}$$
Thus, for each $t\in I_T(H^1,s_1,r)\cap I_T(H^2,s_2,r)$, $F(t)$ is either $F^1$ or $F^2$.
Note that since $g^{(1)}$ and $g^{(2)}$ are fixed, there are only finitely many choices of $F(t)$, and by  Lemma \ref{l36}, each possible $F(t)$ can only generate finitely many intervals of non-intersection. Thus there exist times $t_1,\dots,t_m$ such that
$$I_T(H^1,s_1,r)\cap I_T(H^2,s_2,r)\subseteq\bigcup_{i=1}^m I_T(F(t_i),t_i,r).$$
\end{proof}

\begin{corollary}\label{c52} Fix $x\in X$. Let $H^0,H^1,\dots H^q\in \CPN_p$ be  such that $\hull(H^0)=\hull(H^1)=\dots=\hull(H^q)$ and $H^0$ is distinct from 
$H^i$ for all $i>0$. Let  $T>0$, $0<r<r_0$,  and $s_0,s_1,\dots, s_q\in[0,T]$ be 
such that 
\begin{equation*} I_T(H^0,s_0,r)\subseteq \bigcup_{i=1}^q I_T(H^i,s_i,r).
\end{equation*}
Then there exist a finite set $\mathcal{F}\subseteq\CPN_{p-1}\times[0,T]$ such that 
\begin{equation*} I_T(H^0,s_0,r)\subseteq I_T(\mathcal{F},r).
\end{equation*}
\end{corollary}
\begin{proof} By Lemma \ref{p52}, for each $j\in\{1,\dots,q\}$, there exists a finite set  $\mathcal{F}^j\subseteq\CPN_{p-1}\times[0,T]$
 such that 
$$I_T(H^0,s_0,r)\cap I_T(H^j,s_j,r)\subseteq I_T(\mathcal{F}^j,r).$$
Then 
\begin{equation*}
 I_T(H^0,s_0,r)\subseteq\bigcup_{j=1}^q I_T(H^0,s_0,r)\cap I_T(H^j,s_j,r)
\subseteq\bigcup_{j=1}^qI_T(\mathcal{F}^j,r).
\end{equation*}

\end{proof}

\begin{corollary}\label{c53} Fix $x\in X$. 
Let $H^1,H^2\in\CPN_1$ be such that $\hull(H^1)=\hull(H^2)$. Let 
$s_1,s_2>0$ and $0<r\leqslant \delta$. Then one of the following holds
\enumerate
\item $I_T(H^1,s_1,r)=I_T(H^2,s_2,r)$ 
\item $I_T(H^1,s_1,r)\cap I_T(H^2,s_2,r)=\emptyset$
\end{corollary}
\begin{proof}If $H^1=H^2$ and $t\in I(H^1,s_1,r)\cap I(H^2,s_2,r)\not=\emptyset$, then by (\ref{same}), $$I_T(H^1,s_1,r)=I_T(H^1,t,r)=I_T
(H^2,s_2,r).$$ If  $H^1\neq H^2$, by Proposition \ref{p52}, $I_T(H^1,s_1,r)\cap I_T(H^2,s_2,r)\subseteq I_T(\mathcal{F},r)=\emptyset$, where $\mathcal{F}=\PN_0\times[0,T]=\{e\} \times [0,T]$.\\
\end{proof}

\begin{proposition}\label{p53} Fix $x\in X$, $T>0$, and $0<r\leqslant r_0$.  Let $J=\{t\in [0,T]: G_{u_tx}\cap B_r\neq \{e\}\}$. 
 There exist $H^1,H^2,\dots, H^m\in\CPN_n$ and pairwise disjoint intervals $I^1,I^2,\dots, I^m$ such that for every $i\in \{1,\dots,m\}$, $H^j$ is $(r,u_t,x)$-dominant on  $I^j$ and $\overline{J}=\cup_{j=1}^M I^j$.
\end{proposition}
\begin{proof} Fix $t\in \overline{J}$. By Lemma \ref{l51} and the fact that $B_r$ is a Zassenhaus neighborhood 
there exists a group $H(t)$ which is conjugate to $H_1\times H_2\times...\times H_n$ where each $H_i$ is either  
$M_iA_i$ or $M_iN_i$ and 
$$G_{u_tx}\cap B_{r}\subseteq u_tH(t)u_{-t}.$$
By continuity there exists an open interval $I(t)$ containing $t$ such that $H(t)$ is $(r,u_t,x)$-dominant on $I(t)$. By compactness of $\overline{J}$, a finite number of 
such intervals will cover $\overline{J}$. We may then shrink the intervals if needed to insure that $\overline{J}=\cup_{j=1}^m I^j$ and that the intervals are disjoint. Our intervals may or may not contain their end points.\\
\end{proof}

For the sake of clarity, we first present the proof of the first half of Theorem \ref{maintheorem2} in the case when $G$ is the product of 2 semisimple $\mathbb{R}$-rank 1 groups.  Most of the ideas of the proof  of Theorem \ref{maintheorem2} appear in this case and notation is simpler.\\
\subsection{Quantitative non-divergence in the case $G=G_1\times G_2$.}
\begin{proof} Let $\delta<r_0$, fix $x\in X_\delta$, fix $T>0$ and  fix a one-parameter unipotent subgroup  $\{u_t\}_{t\in\mathbb{R}}$ of $G$. We show the conclusion holds with $c_{\epsilon}=\overline{c}_{\epsilon/4}^2$, where $\overline{c}_{\epsilon}$ is as in Proposition \ref{p51}.
 Since $x\in X_\delta$,  it follows that $G_{u_{0}x}\cap B_r=\{e\}$. Thus by Proposition \ref{p51},  for any $0<r\leqslant\delta$, $H\in\CPN$, and  $s\in[0,T]$,
\begin{equation}\label{e52}
  m\Big(\big\{t\in I_T(H,s,r):G_{u_tx}\cap u_tHu_{-t}\cap B_{\overline{c}_{\epsilon/4} r}\neq \{e\}\big\}\Big)\leqslant\frac{\epsilon}{4}\cdot m( I_T(H,s,r)).
 \end{equation}   Equation (\ref{e52}) is the main tool used in this proof and follows from the work of Kleinbock and Margulis \cite{K}\cite{KM}.
\\

 Let $J=\big\{t\in[0,T]:G_{u_tx}\cap B_{\overline{c}_{\epsilon/4}^2\delta}\neq\{e\}\big\}$. By Proposition \ref{p53} There exists a finite set
 $\{H^j\}\subseteq \CPN_2$  and  a corresponding collection of intervals $\{I_j\}$ such that  $H^j$ is $(\overline{c}_{\epsilon/4}^2\delta,u_t,x)$-dominant on $I_j$ and $\overline{J}=\cup_{j} I_j$.  By construction of the intervals $I_j$,  for each $j$ there exists $t_j\in I_j$ such that $$ G_{u_{t_j}x}\cap u_{t_j}H^ju_{-t_j}\cap B_{\overline{c}_{\epsilon/4}\delta}\neq\{e\}.$$
 
  Suppose $H^i=H^j$, $i\neq j$,  and $I_T\left(H^i,t_i,\overline{c}_{\epsilon/4}\delta\right)\cap I_T\left(H^j,t_j,\overline{c}_{\epsilon/4}\delta\right)\neq\emptyset$, then by equation (\ref{same}),
  $$I_T\left(H^i,t_i,\overline{c}_{\epsilon/4}\delta\right)= I_T\left(H^j,t_j,\overline{c}_{\epsilon/4}\delta\right).$$
   In this case 
 \begin{equation}\label{jcont}m( I_i\cup I_j)\leqslant m(\{t\in I_T(H^i,t_i,\overline{c}_{\epsilon/4}\delta):  G_{u_tx}\cap u_{t}H^iu_{-t}\cap  B_{\overline{c}_{\epsilon/4}^2 \delta}\neq \{e\}\}).\end{equation}
 
 Let 
 $$S=\left\{i:\text{ for all $j<i$ either $H_i\neq H_j$ or } I_T\left(H^i,t_i,\overline{c}_{\epsilon/4}\delta\right)\neq I_T\left(H^j,t_j,\overline{c}_{\epsilon/4}\delta\right)\right\}.$$
 $S$ discards duplicate maximal intervals of $\overline{c}_{\epsilon/4}\delta$-intersection arising from the same \PN-subgroup, so  by equation (\ref{jcont})
  \begin{eqnarray}
 m\left(J\right)&=&
  m\left(\bigcup_{i} \left\{t\in I_i: G_{u_tx}\cap B_{\overline{c}_{\epsilon/4}^2\delta}\neq\{e\}\right\}\right) \\&  \leqslant&  m\left(\bigcup_{i\in S}\left\{t\in I_T\left(H^i,t_i,\overline{c}_{\epsilon/4}\delta\right): G_{u_tx}\cap u_tH^iu_{-t}\cap  B_{\overline{c}_{\epsilon/4}^2\delta}\neq\{e\}\right\}\right).\label{e53}
  \end{eqnarray}

As $S$ is finite  there exists   $\tilde{S}\subset S$ such that
 $$\bigcup_{j\in {S}} I_T\left(H^j,t_j,\overline{c}_{\epsilon/4}\delta\right)=\bigcup_{j\in \tilde{S}} I_T\left(H^j,t_j,\overline{c}_{\epsilon/4}\delta\right),$$
and for each $t\in [0,T]$, $t$ belongs to at most $2$ intervals $I_T\left(H^j,t_j,\overline{c}_{\epsilon/4}\delta\right)$ with $j\in \tilde{S}$ (some $t\in [0,T]$ may not be covered). Thus 
\begin{equation}\label{e54}\sum_{j\in\tilde{S}} m\left( I_T\left(H^j,t_j,\overline{c}_{\epsilon/4}\delta\right)\right)\leqslant2T.\end{equation}

Let $j\in S\setminus\tilde{S}$.
For each $i\in\tilde{S}$, either $H^i$ and $H^j$ are distinct or $I_T\left(H^i,t_i,\overline{c}_{\epsilon/4}\delta\right)$ and $I_T\left(H^j,t_j,\overline{c}_{\epsilon/4}\delta\right)$ are disjoint.  So if 
$\tilde{S}_j:=\{i\in\tilde{S}:H^i\neq H^j\}$,
then

\begin{equation*}
I_T\left(H^j,t_j,\overline{c}_{\epsilon/4}\delta\right)\subseteq\bigcup_{i\in \tilde{S}_j} I_T\left(H^i,t_i,\overline{c}_{\epsilon/4}\delta\right).\tag{C}
\end{equation*}

Therefore, by Corollary \ref{c52}, there exists a finite set $\{F^i\}\in\CPN_1$  and $\{s_i\}\in[0,T]$ such that 
\begin{equation*}
\bigcup_{j\in S\setminus \tilde{S}}I_T(H^j,t_j,\overline{c}_{\epsilon/4}\delta)\subseteq \bigcup_{i} I_T(F^i,s_i,\overline{c}_{\epsilon/4}\delta).\tag{D}
\end{equation*}
Define
$$S'=\left\{i: I_T(F^i,s_i,\delta)\neq I_T(F^j,s_j,\delta)\text{ for all }j<i\,\text{with }\, \hull(F_i)=\hull(F_j)\right\}$$
Analogously to equation (\ref{e53}),

\begin{eqnarray}
m\left( \bigcup_{i\in S\setminus\tilde{S}} I_T(H^i,t_i,\overline{c}_{\epsilon/4}\delta)\right)\!\!\!\!\!&\leqslant&\!\!\!\!\!m\left(\bigcup_{i} I_T\left(F^i,t_i,\overline{c}_{\epsilon/4}\delta\right)\right)\\&\leqslant&\!\!\!\!\!\sum_{i\in {S}'}m\left(\left\{t\in I_T(F^i,s_i,\delta): G_{u_tx}\cap u_{t}F^iu_{-t}\cap B_{\overline{c}_{\epsilon/4}\delta}\neq\{e\}\right\}\right).\label{e55}
 \end{eqnarray}

By Corollary \ref{c53}, for $i,j\in S' $ such that $i\neq j$ and $\hull (F^i)=\hull(F^j)$, we have that $I_T(F^i,s_i,\delta)$ and $ I_T(F^j,s_j,\delta)$ are disjoint. Hence,
 \begin{equation}\label{e56}\sum_{i\in{S}'}m\left(I_T(F^i,s_i,\delta)\right)\leqslant2T.
 \end{equation}
 For $r=\delta$ and $r=\overline{c}_{\epsilon/4}\delta$, by equation (\ref{e52}),
\begin{equation}\label{e57}m\left(\left\{t\in I_T(H^i,t_i,\overline{c}_{\epsilon/4}\delta): G_{u_tx}\cap u_tH^iu_{-t}\cap B_{\overline{c}_{\epsilon/4}^2\delta}\neq\{e\}\right\}\right)\leqslant\frac{\epsilon}{4} m\big(I_T(H^i,t_i,\overline{c}_{\epsilon/4}\delta)\big),\end{equation}
and
\begin{equation}\label{e58}m\left(\left\{t\in I_T(F^i,s_i,\delta): G_{u_tx}\cap u_tF^iu_{-t}\cap B_{\overline{c}_{\epsilon/4}\delta}\neq\{e\}\right\}\right)\leqslant\frac{\epsilon}4 m\big(I_T(F^i,s_i,\delta)\big).\end{equation}

Thus
\begin{align*}
 m\left(J\right)
  &\leqslant   m\left(\bigcup_{i\in S}\left\{t\in I_T(H^i,t_i,\overline{c}_{\epsilon/4}\delta ): G_{u_tx}\cap u_tH^iu_{-t}\cap B_{\overline{c}_{\epsilon/4}^2\delta}\neq\{e\}\right\}\right)&&\text{(Equation (\ref{e53}))}\\
  &\leqslant \sum_{i\in \tilde{S}} m\left( \left\{t\in I_T(H^i,t_i,\overline{c}_{\epsilon/4}\delta): G_{u_tx}\cap u_tH^iu_{-t}\cap B_{\overline{c}_{\epsilon/4}^2\delta}\neq\{e\}\right\}\right)\\&\,\,\,\,\,+ m\left(\bigcup_{i\in S\setminus\tilde{S}}\left\{t\in I_T(H^i,t_i,\overline{c}_{\epsilon/4}\delta): G_{u_tx}\cap u_tH^iu_{-t} \cap B_{\overline{c}_{\epsilon/4}^2\delta}\neq\{e\}\right\}\right)\\
  &\leqslant  \sum_{i\in \tilde{S}}\frac\epsilon4 m\left( I_T(H^i,t_i,\overline{c}_{\epsilon/4}\delta)\right)+ m\left( \bigcup_{i\in S\setminus\tilde{S}}I_T(H^i,t_i,\overline{c}_{\epsilon/4}\delta)\right)&&\text{(Equation  (\ref{e57}))}\\
    &\leqslant \frac{\epsilon}{2}T+\sum_{i\in S'}m\left(\left\{t\in I_T(F^i,s_i,\delta): G_{u_tx}\cap u_tF^iu_{-t}\cap B_{\overline{c}_{\epsilon/4}\delta}\neq\{e\}\right\}\right)&&\text{(Equations  (\ref{e54}) and (\ref{e55}))}
\\
 &\leqslant \frac{\epsilon}{2}T+\sum_{i\in S'}\frac\epsilon4 m\left( I_T(F^i,s_i,\delta)\right) &&\text{(Equation  (\ref{e58}))}
\\
 &\leqslant \frac{\epsilon}{2}T+\frac{\epsilon}{2}T=\epsilon T.&&\text{(Equation  (\ref{e56}))}
\\
 \end{align*}

\end{proof}

\subsection{Quantitative non-divergence in the case $G=G_1\times\dots\times G_n$}
\begin{proof}Let $0<\delta<r_0$. Fix $x\in X_\delta$, fix $T>0$ and  fix a one-parameter unipotent subgroup  $\{u_t\}_{t\in\mathbb{R}}$ of $G$. We need the following proposition.

\begin{proposition}\label{p54}
Let $\overline{c}_\epsilon$ be as in Proposition \ref{p51}, and for $p=0,\dots,n$, let $\delta_p=\overline{c}_\epsilon^p\delta$.
  Let $$J=\big\{t\in
[0,T]:G_{u_tx}\cap B_{\overline{c}_\epsilon^n\delta}\neq\{e\}\big\}.$$ For every $1\leqslant p\leqslant n$ there exist finite sets    $\mathcal{T}^p,\mathcal{F}^p,\mathcal{\tilde{F}}
^p\subset\CPN_p\times[0,T]$ such that
\begin{enumerate}
\item[1)] ${J}\subseteq I_T(\mathcal{T}^n,\delta_n)$
\item[2)] $\mathcal{F}^p,\tilde{\mathcal{F}}^p\subseteq \mathcal{T}^p$, and $\tilde{\mathcal{F}}^1=\emptyset$.
\item[3)]$I_T(\mathcal{F}^p,\delta_{p})\cup I_T(\mathcal{\tilde{F}}^{p},\delta_{p})=I_T(\mathcal{T}^p,\delta_{p})$
\item[4)] For every $t\in [0,T]$, $t$ belongs to  at most $2{n\choose p}$ elements of $\left\{I_T(H,t,\delta_{p-1})\right\}_{(H,t)\in \mathcal{F}^p}$
\item[5)]  If additionally $p\neq 1,$ $m\left(I_T(\mathcal{\tilde{F}}^{p},\delta_{p-1})\right)\leqslant m\left( I_T(\mathcal{T}^{p-1},\delta_{p-1})\right).$
 \end{enumerate}
\end{proposition}

\begin{proof}[Proof of Proposition \ref{p54}] 
Let us construct $\mathcal{T}^n$.
Let $H_1,\dots, H_l\in\CPN_n$  and intervals $I_1,\dots,I_l$ contained in $[0,T]$ satisfy Proposition \ref{p53} with $r=\delta_n$.  For each $j\in\{1,\dots,l\}$, fix $t_j\in I_j$.
By the construction,
\begin{equation}\label{jcont2}{J}\subseteq\bigcup_{i=1}^l I_T\left(H_i,t_i,\delta_n\right).\end{equation}
Then $1)$ holds for $\mathcal{T}^n=\{(H_j,t_j)\}_{j=1}^l$.

In the following, we shall use backward induction to construct $\mathcal{F}^p,\tilde{\mathcal{F}}^p$ and $\mathcal{T}^{p}$ for $p=n,n-1,\dots,2$. Let $2<p\leqslant n$ and assume that (1) $\mathcal{T}^m$ has been constructed for $m=p,\dots, n$, and (2) $\mathcal{F}^m$ and $\mathcal{\tilde{F}}^m$ have been constructed for 
 $m=p+1,\dots,n$ if $p<n$.
Partition the collection $\mathcal{T}^p$ into ${n\choose p}$ sub-collections such that  if $\hull(H_i)=\hull(H_j)$, then $(H_i,t_i)$ and $(H_j,t_j)$ belong to the same sub-collection. 
Arbitrarily index these sub-collections as
$\mathcal{T}^{p}_{1},\mathcal{T}^{p}_{2},\dots,\mathcal{T}^{p}_{n\choose p},$
and pick any one of these, say, $\mathcal{T}^{p}_{j}$. 
Then because of the finiteness of $\mathcal{T}^{p}_{j}$, there exists a further sub-collection 
$\mathcal{F}^p_j\subseteq\mathcal{T}^p_j$ such that for every $t\in [0,T]$, $t$  belongs to at most two elements of $\left\{I_T\left(F,s,\delta_{p-1}\right)\right\}_{\left(F,s\right)\in \mathcal{F}^{p}_{j}}$, 
and 

\begin{equation}\label{noloss}
\bigcup\limits_{\left(F,s\right)\in \mathcal{F}^{p}_{j}}I_T\left(F,s,\delta_{p-1}\right)=\bigcup\limits_{\left(F,s\right)\in \mathcal{T}^{p}_{j}}I_T\left(F,s,\delta_{p-1}\right).
\end{equation}

Define $$\mathcal{\tilde{F}}_j^{p}:=\left\{(H,t)\in\mathcal{{T}}_j^p\setminus\mathcal{F}_j^p:\text{ $\forall$ $(F,s)\in \mathcal{{F}}_j^p$, if $F=H$, then $I_T(F,s,\delta_{p})\neq I_T(H,t,\delta_{p})$}\right\}.$$

Then for $\tilde{\mathcal{F}_j^p}$, repeating the argument as in the proof of $G_1\times G_2$ case (see equations (C) and (D)), by Corollary \ref{c52} there exists a finite set $\mathcal{S}^{p-1}_j\in\CPN_{p-1}\times[0,T]$ such
 that 
\begin{equation}\label{scont} I_T\left(\mathcal{\tilde{F}}^{p}_j,\delta_{p-1}\right)\subseteq I_T\left(\mathcal{{S}}^{p-1}_j,\delta_{p-1}\right).\end{equation}

Now let $\mathcal{T}^{p-1}=\bigcup_{j=1}^{n\choose p}\mathcal{S}^{p-1}_j$, $\mathcal{F}^p=\bigcup_{j=1}^{n\choose p}\mathcal{F}^p_j$, and $\mathcal{\tilde{F}}^p=\bigcup_{j=1}^{n\choose p}\mathcal{\tilde{F}}^p_j$. By the definition of $\mathcal{{F}}^{p}$ and $\mathcal{\tilde{F}}^{p}$, $2)$  holds. By definition of $\tilde{\mathcal{F}}^p$, $3)$ holds. 
By the definition of $\mathcal{F}^p_j$, $4)$ holds. Finally, by
definition of $\mathcal{T}^{p-1}$ and (\ref{scont}), we have that 5) holds. Thus by induction we have constructed $\mathcal{T}^p$ for $p=1,\dots,n$ and $\mathcal{F}^p$ and $\tilde{\mathcal{F}}^p$ for $p=2,\dots,n$. Let $\mathcal{F}^1=\mathcal{T}^1$ and $\tilde{\mathcal{F}}^1=\emptyset$. By definition,  2), 3), 5) follow easily and $4)$ holds due to Corollary \ref{c53}.

\end{proof}

Now we are in the position to prove the quantitative non-divergence in the case when $G=G_1\times\dots\times G_n$.
Recall that $x\in X_{\delta_0}$. By Proposition \ref{p51}, for $j=1,\dots,n$, $H\in\CPN$, and $t\in[0,T]$,
\begin{eqnarray}
m\left(I_T\left(H,t,\delta_j\right)\right)&\leqslant&m\left(\left\{t\in I_T\left(H,t,\delta_{j-1}\right):G_{u_tx}\cap B_{\delta_j}\neq\{e\}\right\}\right)\\
&\leqslant& \epsilon m\left(I_T\left(H,t,\delta_{j-1}\right)\right)\label{klein}
\end{eqnarray}

By Proposition \ref{p54} \#3 and \#5, for $p=2,\dots,n$
 \begin{align}
 m\left(I_T(\mathcal{T}^p,\delta_p)\right)
&\leqslant m\left(I_T(\mathcal{F}^p,\delta_p)\right)+m\left(I_T(\mathcal{\tilde{F}}^{p},\delta_p)\right)\\
&\leqslant\sum_{(H,t)\in\mathcal{F}^p} m\left(I_T\left(H,t,\delta_p\right)\right)+m\left(I_T(\mathcal{\tilde{F}}^p,\delta_{p-1})\right)&&\\
&\leqslant\sum_{(H,t)\in\mathcal{F}^p} m\left(I_T\left(H,t,\delta_p\right)\right)+m\left(I_T(\mathcal{T}^{p-1},\delta_{p-1})\right)\label{induct}
 \end{align}
 
 In the following final step, we shall use (\ref{induct}) inductively for $p=n,n-1,\dots,2$. We have
 
 \begin{align*}
m(J)&\leqslant m\left(I_T(\mathcal{T}^n,\delta_n)\right)&&\text{(Prop. \ref{p54} \#1.)}\\
&\leqslant\sum_{(H,t)\in\mathcal{F}^n} m\left(I_T\left(H,t,\delta_n\right)\right)+m\left(I_T(\mathcal{T}^{n-1},\delta_{n-1})\right)&&\text{(Equation (\ref{induct}))}\\
&\leqslant\sum_{p=1}^n\sum_{(H,t)\in\mathcal{F}^p} m\left(I_T\left(H,t,\delta_p\right)\right)&&\text{(Induction (\ref{induct}))}\\
&\leqslant\epsilon\sum_{p=1}^n\sum_{(H,t)\in\mathcal{F}^p} m\left(I_T\left(H,t,\delta_{p-1}\right)\right)&&\text{(Equation (\ref{klein}))}\\
&\leqslant\epsilon\sum_{p=1}^n2T{n\choose p}=2(2^n-1)\epsilon T.&&\text{(Prop. \ref{p54} \# 4.)}
\end{align*}
Thus the analogue of Theorem \ref{uniform} for the product case holds
for
$c_\epsilon=\left(\overline{c}_{\epsilon/\left(2(2^n-1)\right)}\right)^n$. This
proves the first part of Theorem \ref{maintheorem2}.
\end{proof}

\subsection{Uniform non-divergence in the case $G=G_1\times\dots\times G_n$}



\begin{proof}
Let $\epsilon>0$ be given. Fix $r_0>0$ such that $B_{r_0}$ is a Zassenhaus neighborhood and let $\eta\in(0,1]$. By the proof of Theorem \ref{maintheorem2} there exists $c_{\epsilon/2}>0$ such that if $x=g\Gamma \in X_{\eta r_0}$, then for any unipotent one parameter subgroup $\{u_t\}_{t\in\mathbb{R}}$ of $G$,
\begin{equation*}
\frac{1}{T}m\left(\left\{t\in[0,T]: u_tx\not\in X_{c_{\epsilon/2}\eta r_0}\right\}\right)<\frac12\epsilon, \,\,\text{for all $T\geqslant0$}.
\end{equation*}
Furthermore, if $x$ and $t_0$ are such that $u_{t_0}x\in X_{\eta r_0}$, then for $T> t_0(2\epsilon^{-1}-1)$
\begin{equation*}
\frac{1}{T}m\left(\left\{t\in[0,T]: u_tx\not\in X_{c_{\epsilon/2}\eta r_0}\right\}\right)<\epsilon.
\end{equation*}
We will prove the theorem for  $\delta_\epsilon=c_{\epsilon/2} r_0$. Let $\delta=\eta\delta_\epsilon$ and let $x\in X$. We may assume that  $u_tx\not\in X_{\eta r_0}$ for all $t>0$, otherwise, by above, we are finished. 

 Suppose there exists a \PN-subgroup $H$ such that $I(H,t,\eta r_{0})=[0,\infty)$. By  arguing similarly to the proof the 
 proof of Theorem \ref{uniform}, there exists a constant $\beta_\epsilon=\max\left\{\frac{16C_k^{2k^2}}
 {l_Mc_{\epsilon/2}},\left(\frac{8C_k^{k^2}}{lc_{\epsilon/2}}\right)^2\right\}$, where $k$ is the dimension of 
 $G$, $l_M$ is the constant found in Lemma \ref{l52}, and $C_k=2^kk^3(k^2+1)^\frac{1}{k^2}$,  such that 
 there exists a proper abelian subgroup $L\subseteq H\subseteq G$, $\Delta:=g\Gamma g^{-1}\cap L$ is a 
 lattice in $L$, $u_t$ normalizes $L$, and the covolume of $u_t\Delta u_{-t}$ in $u_tLu_{-t}$ is a fixed constant 
 bounded by $(\beta_\epsilon\delta)^{\dim(L)}$  for all $t\geqslant0$. In fact, the same proof works verbatim, 
 with the exception of the construction of  $\Lambda$. In the current case, we argue as in equations (\ref
 {deltadiv})  and (\ref{falsediv}) to find $\Lambda$ such that  $W:=\rspan(\Lambda)$ is invariant under $
 \Ad_{u_t}$, $\Lambda$ is a lattice in $W$, and $\|\Ad_{u_t}\Lambda\|_0$ is a constant less than $(2C_k^
 {k^2}\eta r_0)^{\rank(\Lambda)}$ for all $t>0$. Then the theorem holds in this case.

Thus we may assume that for any \PN-subgroup $H$ and  any $t>0$ we have $I(H,t,\eta r_0)\neq[0,\infty)$. If we can prove that there exists  $T_0>0$ such that  for any \PN-subgroup $H$ and $t\in[0,T_0]$ 
\begin{equation}\label{T}I\left(H,t,\eta r_0\right)\cap[0,T_0]\neq [0,T_0],\end{equation}
then by Proposition \ref{p51}, the equation (\ref{klein}) holds for $T>T_0$. As Proposition \ref{p54} holds independent of the choice of $x$,  the arguments of section 5.4  hold, and we have
\begin{equation*}
\frac{1}{T}m\left(\left\{t\in[0,T]: u_tx\not\in X_{c_{\epsilon/2}\eta r_0}\right\}\right)<\frac12\epsilon, \,\,\text{for all $T\geqslant T_0$}.
\end{equation*}Thus for this $x$ the trajectory is non-divergent, and we are done.\\

The rest is to prove equation (\ref{T}). First of all, we need several definitions.
 For a  \PN-subgroup $F$, let 
$$t_F:=\min\left\{t\geqslant 0:G_{{u_t}x}\cap u_t Fu_{-t}\cap B_{\eta r_0}=\{e\}\right\}.$$
Since $I(F,t,\eta r_0)\neq[0,\infty)$, it follows that $t_F<\infty$. Moreover,
for a time $t\geqslant 0$, we define  $H_t$  to be the minimal  \PN-subgroup (that is, the \PN-subgroup with the least coordinates) such that $G_{u_tx}\cap B_{\eta r_0}\subseteq H_t$. For two \PN-subgroups $F_1$ and $F_2$, let $\pi_{F_2^\perp}(F_1)$ denote the projection of $F_1$ onto the coordinates on which $F_2$ is trivial.
For a \PN-subgroup $F$, define $L_F$ to be the smallest \PN-subgroup containing both $F$ and $\pi_{F^\perp}(H_{t_F})$. \\

Now we inductively construct collections $\{\mathcal{H}_i\}$ of \PN-subgroups as follows: for the time $t=0$, let 
$$\mathcal{H}_0:=\left\{F\in \CPN:F\subseteq H_0\right\},$$ 
and for $i=1,\dots,n-1$, let
$$\mathcal{H}_i:=\left\{H\in\CPN: \exists\,\, F\in\mathcal{H}_{i-1}\,\,\text{and}\,\,H\subseteq L_F\right\}.$$
 Let $${T_0}=\max\{t_F:F\in\cup_{i=0}^{n-1}\mathcal{H}_i\}.$$
We will show equation (\ref{T}) for this $T_0$. Let $F$ be a \PN-subgroup. Define $F_0=F\cap H_0$ and for $i=1,\dots, n-1$, let $F_i=L_{F_{i-1}}\cap F$. If $F_0=\{e\}$, then $T_F=0$ and equation (\ref{T}) holds. If for some $i=1,\dots,n-1$, we have that $F_i=F_{i-1}$, then by construction $$G_{u_{t_{F_{i-1}}}}\cap B_{r_0}\cap u_{t_{F_{i-1}}}Fu_{-t_{F_{i-1}}}=\{e\}.$$ Since $F_{i-1}\in \mathcal{H}_{i-1}$, it follows that $t_{F_{i-1}}<T_0$ and $F$ satisfies equation ($\ref{T})$. If for every $i=1,\dots,n-1$, $F_i\neq F_{i-1}$, then $F=F_{n-1}$, $F\in \mathcal{H}_{n-1}$, and again $F$ satisfies equation ($\ref{T})$.
\end{proof}


\end{document}